\documentclass[reqno,10pt]{amsart}
\usepackage[english]{babel}
\usepackage{amsmath}

\usepackage{amssymb}
\usepackage[scr=boondoxo]{mathalfa}
\usepackage{graphicx}
\usepackage{epstopdf}
\usepackage{soul}
\usepackage{color}
\sethlcolor{yellow}
\definecolor{rred}{rgb}{0.7,0.0,0.2}
\definecolor{bblue}{rgb}{0.2,0.0,0.7}

\newcommand{\bpr}{\begin{trivlist} \item[]{\bf Proof. }}
\newcommand{\epr}{\hspace*{\fill} $\qed$\end{trivlist}}

\newcommand{\be}{\begin{eqnarray}}
\newcommand{\ee}{\end{eqnarray}}
\newcommand{\ba}{\begin{align}}
\newcommand{\ea}{\end{align}}
\newcommand{\bi}{\begin{itemize}}
\newcommand{\ei}{\end{itemize}}

\newtheorem{theorem}{Theorem}[section]
\newtheorem{proposition}[theorem]{Proposition}
\newtheorem{definition}[theorem]{Definition}
\newtheorem{lemma}[theorem]{Lemma}

\newtheorem{remark}[theorem]{Remark}

\numberwithin{equation}{section}
\newtheorem{method}{Method}
\newcommand{\mthlab}[1]{\label{mth:#1}}
\newcommand{\mthref}[1]{Method~\ref{mth:#1}}
\newcommand\ed[1]{{\color{black} #1}}%
\begin{document}

\title[Blowup for flat slow manifolds]{Blowup for flat slow manifolds}

\author {K. U. Kristiansen} 
\date\today
\maketitle

\vspace* {-2em}
\begin{center}
\begin{tabular}{c}
Department of Applied Mathematics and Computer Science, \\
Technical University of Denmark, \\
2800 Kgs. Lyngby, \\
DK
\end{tabular}
\end{center}

 \begin{abstract}
In this paper we present a method for extending the blowup method, in the formulation of Krupa and Szmolyan, to flat slow manifolds that lose hyperbolicity beyond any algebraic order. Although these manifolds have infinite co-dimension, they do appear naturally in certain settings. For example in (a) the regularization of piecewise smooth systems by $\tanh$, (b) a model of aircraft landing dynamics, and finally (c) in a model of earthquake faulting. We demonstrate the approach on a simple model system and the examples (a) and (b).
\end{abstract}
\section{Introduction}
In this paper we focus on slow-fast systems on $(x,y)\in \mathbb R^{n_s\times n_f}$ of the form
\begin{align}
 x' &=\varepsilon f(x,y,\varepsilon),\label{eq.fast}\\
 y'&=g(x,y,\varepsilon),\nonumber
\end{align}
where $()'=\frac{d}{d\tau}$ denotes differentiation with respect to the fast time $\tau$. System \eqref{eq.fast} is called the fast system. 
Near points where $g(x,y,\varepsilon)=\mathcal O(1)$ the variables $x\in \mathbb R^{n_s}$ and $y\in \mathbb R^{n_f}$ vary on separate time scales for $0<\varepsilon\ll 1$. The variable $y$ is therefore called fast, while $x$ is said to be slow. On the other hand, near points where $g(x,y,\varepsilon)=\mathcal O(\varepsilon)$ then the variables both evolve on the slow time scale $t=\varepsilon \tau$. This is described by the slow system:
\begin{align}
 \dot x &=f(x,y,\varepsilon), \label{eq.slow}\\
 \varepsilon \dot y&=g(x,y,\varepsilon),\nonumber
\end{align}
with $\dot{()}=\frac{d}{dt}$. 
Setting $\varepsilon=0$ in \eqref{eq.fast} and \eqref{eq.slow} gives rise to two different limiting systems: \eqref{eq.fast}$_{\varepsilon=0}$:
\begin{align}
 x' &=0,\label{eq.layer0}\\
 y'&=g(x,y,0),\nonumber
\end{align}
called the \textit{layer problem}, and \eqref{eq.slow}$_{\varepsilon=0}$:
\begin{align}
 \dot x &=f(x,y,0),\label{eq.reduced0}\\
 0&=g(x,y,0),\nonumber
\end{align}
called the \textit{reduced problem}. The set 
\begin{align*}
 C=\{(x,y)\vert g(x,y,0)=0\},
\end{align*}
is called the critical manifold. Eq. \eqref{eq.reduced0} is only defined on the set $C$, while $C$ is a set of critical points of \eqref{eq.layer0}. Subsets $S\subset C$, where the linearization of \eqref{eq.layer0} about $(x,y)\in S$ only has as many eigenvalues with zero real part as there are slow variables $n_s$,  are called \textit{normally hyperbolic}. Due to the special structure of \eqref{eq.layer0} normally hyperbolicity is equivalent to all eigenvalues $\lambda_i$, $i=1,\ldots, n_f$, of $\partial_{y} g(x,y)$, $(x,y)\in S$, satisfying $\text{Re}\,\lambda_i\ne 0$. A normally hyperbolic subset $S$ of $C$ can, by the implicit function theorem applied to $g(x,y)=0$ with $\partial_y g(x,y)\ne 0$, be written as a graph 
\begin{align}
S:\quad y=h_0(x),\quad x\in U.\label{eq.Sgraph}
\end{align}
Fenichel's geometric singular perturbation theory \cite{fen1,fen2} establishes for $0<\varepsilon\ll 1$ (a) the smooth perturbation of $S$ in \eqref{eq.Sgraph} with $U$ compact to
\begin{align*}
 S_\varepsilon:\quad y=h_0(x)+\mathcal O(\varepsilon),
\end{align*}
for $0<\varepsilon\ll 1$, and (b) the existence and smoothness of stable and unstable manifolds of $S_\varepsilon$, tangent at $(x,y)\in S$ for $\varepsilon=0$ to the associated linear spaces of the linearization of \eqref{eq.layer0}. As a consequence, the dynamics of \eqref{eq.slow}, in a vicinity of a normally hyperbolic critical manifold $S$, is accurately described for $0<\varepsilon\ll 1$ as a concatenation of orbits, respecting the direction of time, of the layer problem and the reduced problem. For an extended introduction to the subject of slow-fast theory, the reader is encouraged to consult the references \cite{jones_1995,kaper, kuehn2015}. 

Fenichel's theory does not apply near singular points where $C$ is nonhyperbolic. To deal with such degeneracies, and extend the theory of geometric singular perturbations, the blowup method \cite{dumortier_1991,dumortier_1993,dumortier_1996}, in particular in the formulation of Krupa and Szmolyan \cite{krupa_extending_2001, krupa_extending2_2001,krupa_relaxation_2001}, have proven extremely useful. The method \textit{blows up} the singularity to a higher dimensional geometric object, such as a sphere or a cylinder. By appropriately choosing weights associated to the transformation, it is in many situations possible to divide the resulting vector-field by a power of a polar-like variable measuring the distance to the singularity. This gives rise to a new vector-field, only equivalent to the original one away from the singularity, for which hyperbolicity has been (partially) gained on the blowup of the singularity. Sometimes this approach of blowing up singularities has to be used successively, see e.g. \cite{elena, kosiuk2011a}. 
In this paper we study situations where the blowup approach does not apply directly.

In combination Fenichel's geometric singular perturbation theory and the blowup method have been very successful in describing global phenomena in slow-fast models. A frequently occuring phenomena in such models are relaxation oscillations, characterized by repeated switching of slow and fast motions. A prototypical system where this type of periodic orbit occurs is the van Pol system \cite{krupa_relaxation_2001,kuehn2015} but they also appear in many other models, in particular in those arising from neuroscience \cite{izhi}. 

Recently more complicated examples of relaxation oscillations have been studied using these geometric methods. In \cite{kosiuk2015a}, for example, the authors consider a model for the embryonic cell division cycle at the molecular level in eukaryotes. This model is an example of a slow-fast system not in the standard form \eqref{eq.fast}: The slow-fast behaviour is in some sense hidden. Furthermore, the chemistry imposes particular nonlinearities that give rise to special self-intersections of the set of critical points. Using the blowup method, the authors show that this in turn results in a novel type of relaxation oscillation in which segments of the periodic orbit glue close to the self-intersections of the critical set. In \cite{Gucwa2009783} another interesting oscillatory phenomenon is studied involving an unbounded critical manifold. Finally, the reference \cite{kosiuk2011a} provides a detailed description of relaxation oscillations occurring in a model describing glycolytic oscillations with two small singular parameters. The references \cite{Gucwa2009783} and \cite{kosiuk2011a} both apply blowup methods. These examples illustrate that each problem is unique;  models have their own peculiar degeneracies, and a unifying framework is not possible. Yet the blowup method provides a foundation from which these systems can be dealt with rigorously. However, recently in \cite{elena}, the authors consider a model from \cite{erickson2008a}:
\begin{align}
 \dot x&=e^{z}(x+(1+\alpha)z),\label{eq.elena}\\
 \dot y &= e^z-1,\nonumber\\
 \varepsilon \dot z&=-e^{-z}\left(y+\frac{x+z}{\xi}\right),\nonumber
\end{align}
(using the variables introduced in \cite[Eq. (1)]{elena}), 
describing earthquake faulting, for which the blowup method, in its original formulation, fails. We will explain why in the following.

The system \eqref{eq.elena} has a degenerate Hopf bifurcation at $\alpha=\xi,\,\varepsilon=0$ within an everywhere attracting, but unbounded, critical manifold
\begin{align}
 y+\frac{x+z}{\xi}=0.\label{eq.elenaC}
\end{align}
A vertical family of periodic orbits emerges from this bifurcation for $\varepsilon=0$. The analysis of the perturbation of these periodic orbits is initially complicated by the loss of compactness. Using Poincar\'e compactification \cite{chicone}, the authors study the critical manifold at infinity. There the critical manifold is shown to lose normal hyperbolicity. This is due to the fact that the non-trivial eigenvalue of the linearization of \eqref{eq.elena}$_{\varepsilon=0}$ about \eqref{eq.elenaC} decays exponentially as $z\rightarrow \infty$. The blowup method requires the homogeneity of algebraic terms to leading order, and this approach does therefore not directly apply to the exponential loss of hyperbolicity that occurs in this model. Applying the method presented in the present paper, the authors of \cite{elena} nevertheless managed to obtain a new geometric insight into the peculiar relaxation oscillations that occur in this model. In particular, a locally invariant manifold was found, different from the critical manifold of the system, and not directly visible prior to blowup, which organizes the dynamics at infinity. The basic idea of the method in the present paper is to embed the system into a higher dimensional model. \ed{Increasing the phase space dimension is already central to the original blowup formulation of Krupa and Szmolyan as the small parameter is always included in the blowup. In our method, we will augment a new dynamic variable in such a way that the resulting system is algebraic to leading order at the degeneracy and therefore (potentially) amendable to blowup. }

%
%
%
Another example,  where blowup does not apply (or seem to be useful), is a slow-fast system undergoing a dynamic Hopf bifurcation. Similar problems occur in Hamiltonian systems with fast oscillatory behaviour. Such systems are studied in \cite{geller1,2012arXiv1208.4219U,kristiansen2015a} using separate techniques. However, in the case of dynamic Hopf, it is shown in \cite{hayes2015a} that blowup can be successfully applied when combined with the technique, popularized by Neishtadt in \cite{neishtadt1987a,neishtadt1988a}, of complex time. The reference \cite{usNext} deals with fast oscillations using a different approach. 
Similarly, the blowup approach does not appear to help when applied directly to the problem of bifurcation delay \cite{schecter}. However, recently in \cite{de2016a}, the authors showed that a transformation of the fast variable by a \textit{flat} function, can bring the system into a system that is amendable to blowup.  \ed{The problem of bifurcation delay was also considered in \cite{ting}. Here the author describes this problem in a different way, following an approach similar to the one promoted in the present paper, by extending the phase space dimension by one. This extension  enables the use of known results such as exchange lemmas \cite{schecter2008a}. }

\subsection{Overview}
\ed{ In general, our approach can be summarized as follows:
 \begin{method}\mthlab{method}
  Consider a system $(x,y)\in \mathbb R^{n+1}$ with a small parameter $\epsilon\ge 0$. Suppose that there exists a critical manifold for $\epsilon=0$ with an eigenvalue $\lambda(y)$ that decays exponentially as $y\rightarrow 0^+$. To study the system using blowup (and hyperbolic methods of dynamical systems theory) then do the following:
\begin{itemize}
 \item[Step 1:] Introduce $q=\lambda(y)$ (or a scaling thereof) as a new dependent variable.
 \item[Step 2:] Use implicit differentiation to obtain a differential equation for $q$.
 \item[Step 3:] Substitute $q$ for those expressions in $y$ in the original system that decays exponentially. (Do not substitue for expressions in $y$ that decay algebraically). This gives an extended system $(x,y,q)\in \mathbb R^{n}\times \overline{\mathbb R}_+^2$ that agrees with the original one on the invariant set $\{(x,y,q)\vert q=\lambda(y)\}$. (It may be necessary to multiply the right hand side by powers of $y$ to ensure that the extended system is well-defined at $y=0$.) 
 \item[Step 4:] Consider the system on $(x,y,q,\varepsilon)\in \mathbb R^{n}\times \overline{\mathbb R}_+^3$ and apply a weighted blowup of the variables $(x,q,\epsilon)$ (but not $y$) with $q=\epsilon=0$ .
\end{itemize}
\end{method}}
We do not aim to provide a general geometric framework for our approach. This must be part of future research. Instead we will focus on successful applications of our approach. In section \ref{sec.model0} we first present the general problem and illustrate our approach by considering an extension of a set of models considered by Kuehn in \cite{kuehn2014}. 
 In section \ref{sec.tanh} we then apply the method to study regularization of piecewise smooth (PWS) systems. PWS systems are of great significance in applications \cite{Bernardo08}; they occur in mechanics (friction, impact), in biology (genetic regulatory networks) and in variable structure systems in control engineering \cite{Utkin77}. But these systems also pose many problems, both computationally and mathematically. A frequently used approach is therefore to apply regularization. This has for example been done in the references \cite{krihog,krihog2,krihog3,us,usNext}, to deal with problems associated with lack of uniqueness in PWS systems, and in \cite{reves_regularization_2014}. These references, however, exclude regularization functions such as $\tanh$ due its special asymptotic properties that leads to loss of hyperbolicity at an exponential rate. We will in this paper demonstrate how the approach of this paper can be used to study the regularization by $\tanh$. 

In section \ref{sec.aircraft} we finally consider a model from \cite{aircraft} of aircraft ground dynamics. The model is a $6D$ rigid body model but displays $2D$ slow-fast phenomena such as a canard-like explosion of limit cycles. The authors of \cite{aircraft} do not investigate the origin of the slow-fast structure but present the following $2D$ slow-fast model:
\begin{align}
 \dot u &= -\varepsilon(\alpha-v),\label{eq.aircraftModel00}\\
 \dot v&= -u-(v-a)e^{vb}.\nonumber
\end{align}
see \cite[Eqs. (7) and (8)]{aircraft}, as a \textit{minimal model} capturing the key features. Here $(u,v)$ describes the planar velocity of the center of mass of the aircraft in body fixed coordinates. See \cite[Figs. (3) and (4)]{aircraft}. The critical manifold of this system loses hyperbolicity beyond any algebraic order as $v\rightarrow -\infty$. We will apply the method of this paper to describe the special canard explosion occurring in this model. 
\section{A method for flat slow manifolds}\label{sec.model0}
Kuehn in \cite[Proposition 6.3]{kuehn2014} studies the following slow-fast system
\begin{align}
 \dot u &=\varepsilon \mu,\label{eq.kuehn0}\\
 \dot v&=1-v^n u,\nonumber
\end{align}
with $0<\varepsilon\ll 1$, $\mu\ne 0$, $(u,v)\in \mathbb R^2$, and $n\in \mathbb N$. 
Here we focus on $v\ge c>0$ and $n\ge 2$. In comparison with \cite{kuehn2014} we have also replaced $x,\,y,\,s$ by $v,\,u$ and $n$, respectively.
The set
\begin{align*}
 S_a:\quad u = v^{-n},\quad v\ge c>0,
\end{align*}
is an attracting critical manifold of \eqref{eq.kuehn0}. Indeed the linearization of \eqref{eq.kuehn0} about $S_a$ gives
$-nv^{n-1}u = -n v^{-1}$
as a single non-trivial negative eigenvalue. By Fenichel's theory, fixed compact subsets of $S_a$ smoothly perturb into an attracting, locally invariant slow manifolds $S_{a,\varepsilon}$. The reference \cite{kuehn2014} investigates how far as $v\rightarrow \infty$ the manifold $S_{a,\varepsilon}$ can be extended as a perturbation of $S_a$. For this the author applies the change of coordinates $(u,v)=(x,y^{-1})$ to compactify $v\ge c$. This gives
\begin{align}
 \dot x&=\varepsilon\mu y^n,\label{eq.kuehn1}\\
 \dot y &=y^{2}(x-y^n),\nonumber\\
 \dot \varepsilon &=0,\nonumber
\end{align}
after a nonlinear transformation of time that corresponds to multiplication of the right hand side by $y^n$. \ed{This time transformation ensures that the system is well-defined at $y=0$ and leaves the orbits with $y>0$ unchanged.} In the $(x,y)$-variables, $S_a$ becomes
\begin{align*}
 S_a:\quad x=y^n.
\end{align*}
We continue to use the same symbol for $S_a$ in the new $(x,y)$-variables. The manifold $S_a$ is nonhyperbolic at $y=0$. This is due to the asymptotic alignment of the tangent spaces $TS_a$ for $v\rightarrow \infty$ with the critical fibers in the original $(u,v)$-variables. The result \cite[Proposition 6.3]{kuehn2014} then states that $S_{a,\varepsilon}$ extends within $x\ge 0,y\ge 0$ as a perturbation of $S_a$ up until a neighborhood of $(x,y)=0$ that scales like
\begin{align}
 (x,y) = \left(\mathcal O(\varepsilon^{n/(n+1)}),\mathcal O(\varepsilon^{1/(n+1)})\right),\label{eq.kuehnResult}
\end{align}
with respect to $\varepsilon\rightarrow 0$. 
\subsection{Blowup}
To prove \eqref{eq.kuehnResult}, Kuehn, in \cite{kuehn2014}, applies the following blowup transformation
\begin{align}
 x = r^n \bar x,\quad y = r\bar y,\quad \varepsilon = r^{n+1}\bar\epsilon,\quad (r,(\bar x,\bar y,\bar\epsilon))\in \overline{\mathbb R}_+\times S^2,\label{eq.BUKuehn}
\end{align}
of $(x,y,\varepsilon)=0$. Here $$S^2=\{(\bar x,\bar y,\bar \epsilon)\vert \bar x^2+\bar y^2+\bar \epsilon^2=1\}.$$
We introduce the blowup method by considering this example. Let $B=\overline{\mathbb R}_+\times S^2$ denote the \textit{blowup space}. Then the blowup \eqref{eq.BUKuehn} can be viewed as a mapping:
\begin{align*}
 \Phi:\quad B\rightarrow \mathbb R^3,
\end{align*}
blowing up the nonhyperbolic point $$x=y=\varepsilon=0,$$
to a sphere $(\bar x,\bar y,\bar\epsilon)\in S^2$ within $r=0$. 
The map $\Phi$ transforms the vector-field $X$ in \eqref{eq.kuehn1} to a vector-field $\overline X=\Phi^*(X)$ on $B$ by pull-back. Here $\overline{X}\vert_{r=0}=0$ but the exponents, or weights, of $r$ in the blowup \eqref{eq.FirstBU}, $n$, $1$ and $n+1$, respectively, have been chosen so that $\overline{X}$ has a power of $r$, here $r^{n+1}$, as a common factor. The vector-field can therefore be \textit{desingularized} through the division of $r^{n+1}$. In particular, $\widehat X\equiv r^{-(n+1)} \overline{X}$ is well-defined and non-trivial $\widehat X\vert_{r=0}\ne 0$. To described the dynamics of $\widehat X$ on the blowup space we could use spherical coordinates to describe $S^2$. But as demonstrated in \cite{krupa_extending_2001} the dynamics across the blowup sphere varies significantly in general and it is therefore almost mandatory to use directional charts. \ed{Loosely speaking, we obtain a directional chart for \eqref{eq.BUKuehn}, describing the subset $B\cap \{\bar y>0\}$, by setting 
%
%
$\bar y=1$:
\begin{align}
 x=r_1^n x_1,\quad y=r_1,\quad \varepsilon=r_1^{n+1}\epsilon_1.\label{eq.firstChart}
\end{align}
Setting $\bar \epsilon=1$ in \eqref{eq.BUKuehn} similarly gives the \textit{scaling chart}:
\begin{align}
 x=r_2^n x_2,\quad y=r_2y_2,\quad \varepsilon=r_2^{n+1},\label{eq.temp}
\end{align}
We will use subscripts to distinguish the variables in \eqref{eq.firstChart} and \eqref{eq.temp} from those appearing in \eqref{eq.BUKuehn}. Geometrically \eqref{eq.firstChart}, for example, can be interpreted as a stereographic-like projection from the plane $\{(x_1,1,\epsilon_1)\vert (x_1,\epsilon_1) \in \mathbb R^2\}$, tangent to $S^2$ at $(\bar x,\bar y,\bar \epsilon)=(0,1,0)$, to the upper hemisphere $S^2\cap\{\bar y>0\}$:
\begin{align*}
 x_1 = \frac{\bar x}{\bar y^n},\quad \epsilon_1 = \frac{\bar \epsilon}{\bar y^{n+1}}.
\end{align*}
See also Fig. (\ref{fig.New}) below which illustrates charts used later in the manuscript. We shall therefore frequently abuse notation slightly and simply refer to (\ref{eq.firstChart}) and (\ref{eq.temp}) as the charts  ``$\bar y=1$'' and ``$\bar \epsilon=1$'', respectively. 
 }

Different blowups and charts will appear during the manuscript. We will use the same notation and often identical symbols for each blowup. Although this can potentially lead to confusion we also believe that it stresses the standardization of the method, emphasizing the similarities of the arguments and the geometric constructions. Blowup variables are given a bar, such as $(\bar x,\bar y,\bar\epsilon)$ in \eqref{eq.BUKuehn}. Charts such as \eqref{eq.firstChart} will be denoted by $\kappa_i$, using subscripts to distinguish between the charts and the corresponding local coordinates. Similarly, we will use the (standard) convention that manifolds, sets and other dynamical objects in chart $\kappa_i$ are given a subscript $i$. An object (set, manifold), say $M_i$ obtained in chart $\kappa_i$, will in the blowup variables be denoted by $\overline M$. Finally, if $M_i$ in chart $\kappa_i$ is visible in chart $\kappa_j$ then it will be denoted by $M_j$ in terms of the coordinates in this chart.

\subsection{A flat slow manifold}\label{sec.modelFlat}
Now, we return to \eqref{eq.kuehn1} and ask the following question: What happens if we replace $y^n$ in \eqref{eq.kuehn1} by $e^{-y^{-1}}$? Then  
\begin{align}
 \dot x&=\varepsilon \mu e^{-y^{-1}},\label{eq.kuehnNew}\\
 \dot y &=y^2(x-e^{-y^{-1}}),\nonumber
\end{align}
and
\begin{align}
 S_a:\quad x = e^{-y^{-1}},\label{eq.SauEqExp}
\end{align}
re-defining $S_a$ again, is a critical manifold of \eqref{eq.kuehnNew}. It is \textit{flat} as a graph over $y$ at $y=0$ in the sense that all derivatives of the right hand side of \eqref{eq.SauEqExp} vanish at $y=0$. The linearization of \eqref{eq.kuehnNew} about $S_a$ in \eqref{eq.SauEqExp} gives
\begin{align}
 -e^{-y^{-1}},\label{eq.eigenvalue0}
\end{align}
as a single non-trivial eigenvalue. Hence $S_a$ is therefore attracting for $y>0$ but loses hyperbolicity at an exponential rate as $y\rightarrow 0^+$. Then, as outlined in the introduction, the blowup method does not directly apply. The blowup method requires homogeneity of the leading order terms to enable the desingularization, and letting $n\rightarrow \infty$ in \eqref{eq.kuehn1} and \eqref{eq.kuehnResult} is clearly hopeless. We will in the following demonstrate the use of \mthref{method} by extending the slow manifold of \eqref{eq.kuehnNew} near $y=0$. Generalizations of the approach to more general examples of flat functions, such as $y^{\alpha}e^{-y^{-\beta}}$ with $\beta>0$, are straightforward.

 \subsection{How to deal with flat slow manifolds}
 We consider each of the steps in \mthref{method}.
 \subsection*{Step 1}
 The basic idea of our approach is to augment the exponential:
\begin{align}
 q = e^{-y^{-1}},\label{eq.kuehnQ0}
\end{align}
which is also the negative of the eigenvalue \eqref{eq.eigenvalue0}, as a new dynamic variable. 
\subsection*{Step 2}
Differentiating \eqref{eq.kuehnQ0} gives
\begin{align*}
 \dot q&=e^{-y^{-1}} y^{-2}\dot y = e^{-y^{-1}}(x-e^{-y^{-1}})=q(x-q),
\end{align*}
using \eqref{eq.kuehnNew} in the second equality and \eqref{eq.kuehnQ0} in the third. 
\subsection*{Step 3}
But then using \eqref{eq.kuehnQ0} in \eqref{eq.kuehnNew} we obtain an extended system
\begin{align}
 \dot x&=\varepsilon \mu q,\label{eq.eqkuehnq}\\
 \dot y&=y^2 (x-q),\nonumber\\
 \dot q&=q(x-q),\nonumber\\
 \dot \varepsilon &=0,\nonumber
\end{align}
on $(x,y,q,\varepsilon)\in \mathbb R\times \overline{\mathbb R}_+^3$. 
Introducing $q$ by \eqref{eq.kuehnQ0} automatically embeds a hyperbolic-center structure into the system: To illustrate this simple fact, suppose time is so that $y$ has algebraic, center-like decay as $\mathcal O(1/t)$ for $t\rightarrow \infty$. Then $q$ decays hyperbolically as $\mathcal O(e^{-t})$. This construction will therefore enable the use of center manifold theory and normal form methods to study systems like \eqref{eq.kuehnNew}.

The set 
\begin{align}
 Q=\{(x,y,q,\varepsilon)\vert \,q=e^{-y^{-1}}\},\label{eq.kuehnQ}
\end{align}
is by construction an invariant set of \eqref{eq.eqkuehnq}. However, the invariance of $Q$ is implicit in \eqref{eq.eqkuehnq} and we can invoke it when needed. Now $S_a$ in \eqref{eq.SauEqExp} becomes
\begin{align}
 S_a:\quad x=q,\,\varepsilon=0,\label{eq.SaUEqq}
\end{align}
using the same symbol in the new variables $(x,y,q,\varepsilon)$. 
This is a critical manifold of the extended system \eqref{eq.eqkuehnq} for $\varepsilon=0$. The linearization now has 
$-q$
as a single non-trivial eigenvalue. The two-dimensional manifold $S_a$ in \eqref{eq.SaUEqq} is hyperbolic within $\varepsilon=0$ except along the line $y\ge 0,\,x=q=0$. But now, by construction, the loss of hyperbolicity is algebraic. 
\subsection*{Step 4}
We now apply the following blowup transformation:
\begin{align}
  y=\bar y,\quad x = r \bar x,\quad q = r\bar q,\quad \varepsilon = r\bar\epsilon,\quad (\bar y,r,(\bar x,\bar q,\bar\epsilon))\in \overline{\mathbb R}_+^2\times S^2,\label{eq.FirstBU}
\end{align}
of $(x,q,\varepsilon)=0$. In this case the blowup transformation blows up the line $y\ge 0,\,x=q=\varepsilon=0$ to a cylinder $(\bar y,(\bar x,\bar q,\bar\epsilon))\in \overline{\mathbb R}_+\times S^2$ and desingularization is obtained through division of the resulting vector-field by $r$.
In this section we shall only focus on the following \textit{entry} chart 
\begin{align}
 \kappa_1:\quad \bar q=1:\quad x=r_1x_1,\,q=r_1,\,\varepsilon=r_1 \epsilon_1,\label{eq.chartKappa1Model0}
\end{align}
with $r_1\ge 0$, to cover $S^2\cap\{\bar q>0\}$ of the blowup sphere. Notice that $y$, \ed{in agreement with step 4 of \mthref{method}}, is not transformed by \eqref{eq.FirstBU} and we will therefore for simplicity continue to use this symbol in chart $\kappa_1$ (a convention we frequently follow).
\subsection{Chart $\kappa_1$}
Inserting \eqref{eq.chartKappa1Model0} into \eqref{eq.eqkuehnq} gives
\begin{align}
\dot r_1 &=r_1(x_1-1),\label{eq.eqmodel0}\\
\dot x_1 &=(1-x_1)x_1+\epsilon_1\mu,\nonumber\\
\dot y &=y^2(x_1-1),\nonumber\\
\dot \epsilon_1 &=(1-x_1)\epsilon_1,\nonumber
\end{align}
after division of the right hand side by $r_1$. The set $Q$ in \eqref{eq.kuehnQ} becomes
\begin{align*}
 Q_1=\{(r_1,x_1,y,\epsilon_1)\vert r_1 = e^{-y^{-1}}\}.
\end{align*}
\begin{remark}\label{remark0}
 Notice that the $r_1$-equation decouples in \eqref{eq.eqmodel0}. This is possible in the chart $\bar q=1$ in all the models and settings that I have considered. \ed{It is a consequence of step 3 and the fact that we do not substitute for $q$ in expressions with $y$ that decay algebraically.} For system \eqref{eq.eqkuehnq}, this effectually means that the dimension of the resulting system is the same as the dimension of the chart $\bar y=1$ associated with the blowup in \eqref{eq.BUKuehn} of \eqref{eq.kuehn0}.
 

\end{remark}
Note that $y=\delta$ corresponds to
\begin{align*}
 r_1 = e^{-\delta^{-1}},
\end{align*}
within $Q_1$. Let
\begin{align*}
 \rho(\delta) = e^{-\delta^{-1}}.
\end{align*}
Then we consider the following set
\begin{align*}
 U_1 = \{(r_1,x_1,y,\epsilon_1)\vert r_1\in [0,\rho(\delta)],\,\epsilon_1\in [0,\nu],\,y\in [0,\delta],x_1\in [0,\xi^{-1}]\},
\end{align*}
with $\delta$, $\nu$ and $\xi$ all sufficiently small.

The critical manifold $S_a$ in \eqref{eq.SaUEqq} becomes
\begin{align*}
 S_{a,1}:\quad x_1=1,\,\epsilon_1=0.
\end{align*}
The advantage of the blowup is that we have gained hyperbolicity of $S_{a,1}$ at $r_1=\epsilon_1=0$. Indeed, the linearization of \eqref{eq.eqmodel0} about a point on the line $$L_1:\,x_1=1,\,r_1=0,\,\epsilon_1=0,\,y\ge 0,$$ gives $-1$ as a single non-zero eigenvalue. The associated eigenvector is purely in the $x_1$-direction. The center space is, on the other hand, spanned by two eigenvectors, purely in the direction of $r_1$ and $y$, respectively, and the eigenvector $(0,1,0,\mu^{-1})$. By center manifold theory we therefore directly obtain the following:
\begin{proposition}\label{prop.modelM1}
 Within $U_1$ there exists a center manifold
 \begin{align*}
  M_1:\quad x_1 = 1-\epsilon_1\mu(1+\mathcal O(\epsilon_1)).
 \end{align*}
$M_1$ contains $S_{a,1}$ within $\epsilon_1=0$ as set of equilibria and 
\begin{align}
 C_1:\quad x_1 = 1-\epsilon_1 \mu(1+\mathcal O(\epsilon_1)),\quad r_1=y=0,\label{eq.C10}
\end{align}
contained within $r_1=y=0$, 
as a center sub-manifold. The sub-manifold $C_1$ is overflowing (inflowing) if $\mu>0$ ($\mu<0$). 
\end{proposition}
This result can be viewed as an extension of \cite[Lemma 5.4]{kuehn2014} to a flat slow manifold. The result is illustrated in Fig. \ref{fig.modelM1}. In the following we present some conclusions from this result. 
\subsection{Conclusion}
The center manifold $M_1$ intersects the two invariant sets $Q_1$ and $E_1$. This intersection $M_1\cap Q_1\cap E_1$ is the extension of Fenichel's slow manifold, being $\mathcal O(e^{-c/\varepsilon})$-close to $S_{a,\varepsilon}$ at $r_1=\rho(\delta)$. It intersects $\{\epsilon_1=\nu\}$ with
\begin{align*}
 x_1 = 1-\nu\mu(1+\mathcal O(\nu)),\quad e^{-y^{-1}}=\varepsilon \nu^{-1},
\end{align*}
using the conservation of $Q_1$ and $E_1$. Blowing back down using \eqref{eq.chartKappa1Model0} we realize that we have extended the slow manifold as a center manifold up to 
\begin{align*}
 x = \varepsilon \nu^{-1}(1-\nu\mu(1+\mathcal O(\nu)),\quad y = \ln^{-1} (\varepsilon^{-1} \nu),
\end{align*}
where it is $\ln^{-1} (\varepsilon^{-1} \nu)$-close to $C_1$ in \eqref{eq.C10}. In other words,
$S_{a,\varepsilon}$ extends as a perturbation of $S_a$ up until a neighborhood of $(x,y)=0$ that scales like
\begin{align}
 (x,y) = \left(\mathcal O(\varepsilon),\mathcal O(\ln^{-1} \varepsilon^{-1})\right).\label{eq.kuehnResultNew}
\end{align}
\begin{remark}
It is interesting to note that letting $n\rightarrow \infty$ formally in \eqref{eq.kuehnResult}, describing the extension of $S_{a,\varepsilon}$ for \eqref{eq.kuehn1}, gives
\begin{align*}
 (x,y) = \left(\mathcal O(\varepsilon),\mathcal O(\varepsilon^{0})\right).
\end{align*}
\end{remark}

\begin{figure}[h!t!]\begin{center}{\includegraphics[width=.99\textwidth]{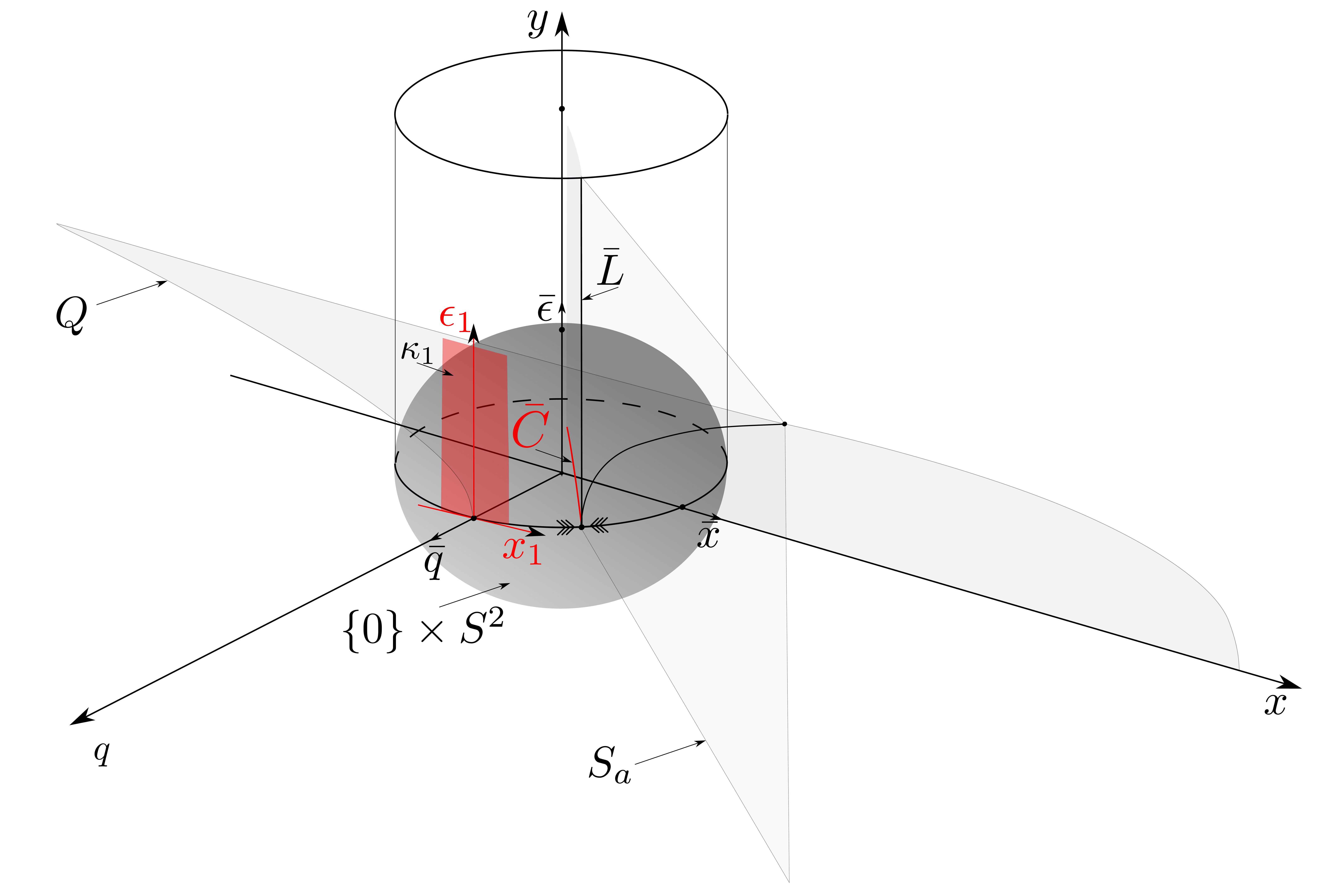}}
\label{fig.modelM1}
\caption{Illustration of the result of Proposition \ref{prop.modelM1} obtained by \mthref{method} and the blowup \eqref{eq.FirstBU}. The transformation \eqref{eq.FirstBU} blows up the line $q=x=\varepsilon=0$ to a cylinder of spheres $(y,(\bar q,\bar x,\bar \epsilon))\in \overline{\mathbb R}_+\times S^2$. We have inserted the sphere $\{0\}\times S^2$ within $y=0$. The blowup enables an extension of $S_a$ onto the blowup sphere $\{0\}\times S^2$ as a normally hyperbolic center manifold $\bar C$ ($C_1$ in chart $\kappa_1$).   }
\end{center}
\end{figure}

To continue $C_1$ across the blowup sphere one may consider the \textit{scaling} chart $$\kappa_2:\quad \bar \varepsilon =1.$$ For system \eqref{eq.eqkuehnq} this chart corresponds to 
 \begin{align*}
  y=y,\,x=r_2x_2,\,q=r_2q_2,\,\varepsilon=r_2.
 \end{align*}
 We will skip the details here for this introductory system and instead focus on our two main examples:
Regularization of PWS systems by $\tanh$ and a model of aircraft ground dynamics, 
 to be considered in the following sections.  However, we will here note that $\epsilon_1 = \nu$ in chart $\bar q=1$ in general, due to the conservation of $\varepsilon$, corresponds to $r_1=o(1)$ with respect to $\varepsilon$. In \eqref{eq.eqkuehnq} we have $r_1=\varepsilon/\nu$ at $\varepsilon=\nu$. Therefore by the invariance of $Q$, we always have $y=o(1)$ (logarithmically with respect to $\varepsilon$ as in \eqref{eq.kuehnResultNew}) in the scaling chart $\kappa_2$ and this will effectively enable the decoupling of $y$ in scaling chart.

\section{Regularization of PWS systems by $\tanh$}\label{sec.tanh}
In this section we shall consider the following planar $(x,y)$ PWS vector-field $X=(X^+,X^-)$ with $X^+=(1,2x)$ defined on $\Sigma^+:\,y>0$:
\begin{align}
 \dot x &=1,\label{eq.yPos}\\
 \dot y&=2x,\nonumber
\end{align}
and $X^-=(0,1)$ defined on $\Sigma^-:\,y<0$:
\begin{align}
\dot x &=0,\label{eq.yNeg}\\
\dot y&=1.\nonumber
\end{align}
This system is a PWS \textit{normal form} for the planar \textit{visible fold}, see \cite[Proposition 3.4]{guardia2011a}.\footnote{In comparison with \cite[Proposition 3.4]{guardia2011a} we have replaced their $(x,y,t)$ by $(2x,2y,2t)$.} The discontinuity set $$\Sigma:\quad y=0,$$ is called the switching manifold and $T=(0,0)\in \Sigma$ is a \textit{visible fold point} since the orbit $y=x^2$ of $X^+$ has a
quadratic tangency with $\Sigma$ at $T$ while $X^-(T)\ne 0$. The point $T\in \Sigma$ divides $\Sigma$ into a (stable) sliding region $$\Sigma_{sl}:\quad x<0,\,y=0$$ and a crossing region $$\Sigma_{cr}:\quad x>0,\,y=0.$$ See Fig. \ref{GeneralSetupR2}. For $p\in \Sigma_{sl}$ the vectors $X^{\pm}(p)$ are in opposition and to continue orbits forward in time one has to define a vector-field $X_{sl}$ on $\Sigma_{sl}$. A natural choice is to follow the Filippov convention and define the sliding vector-field:
\begin{align}
 X_{sl} =\lambda X^++(1-\lambda) X^-,\quad \lambda(x) =\frac{X_2^-(x,0)}{X_2^-(x,0)-X_2^+(x,0)},\label{eq.XSl}
\end{align}
which has a nice geometric interpretation illustrated in Fig. \ref{Filippov}. 
\begin{figure}
\begin{center}
\includegraphics[width=.7\textwidth]{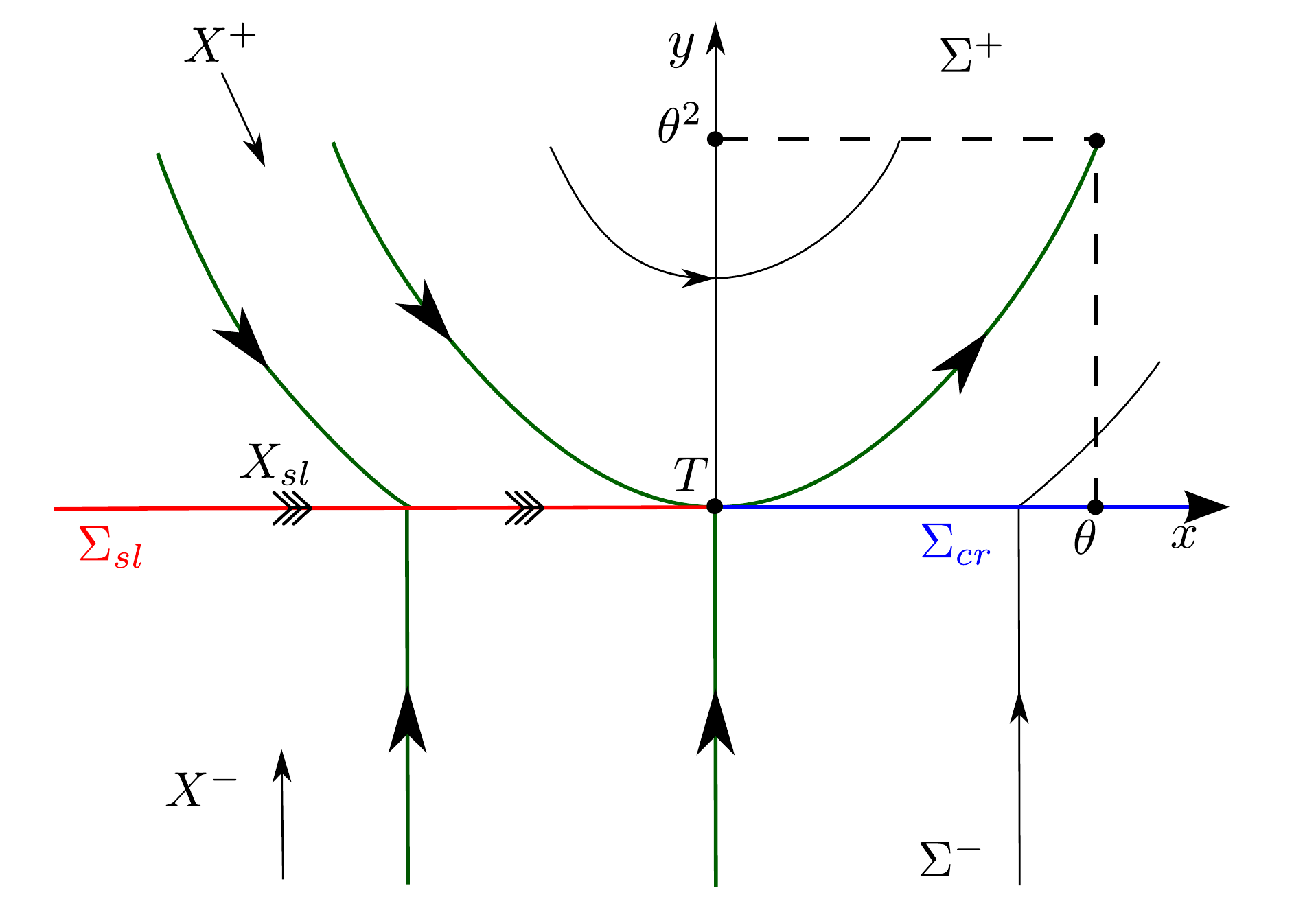}
 \caption{PWS phase portrait of \eqref{eq.yNeg} and \eqref{eq.yPos}.}
\label{GeneralSetupR2}
 \end{center}
              \end{figure}
\begin{figure}
\begin{center}
\includegraphics[width=.7\textwidth]{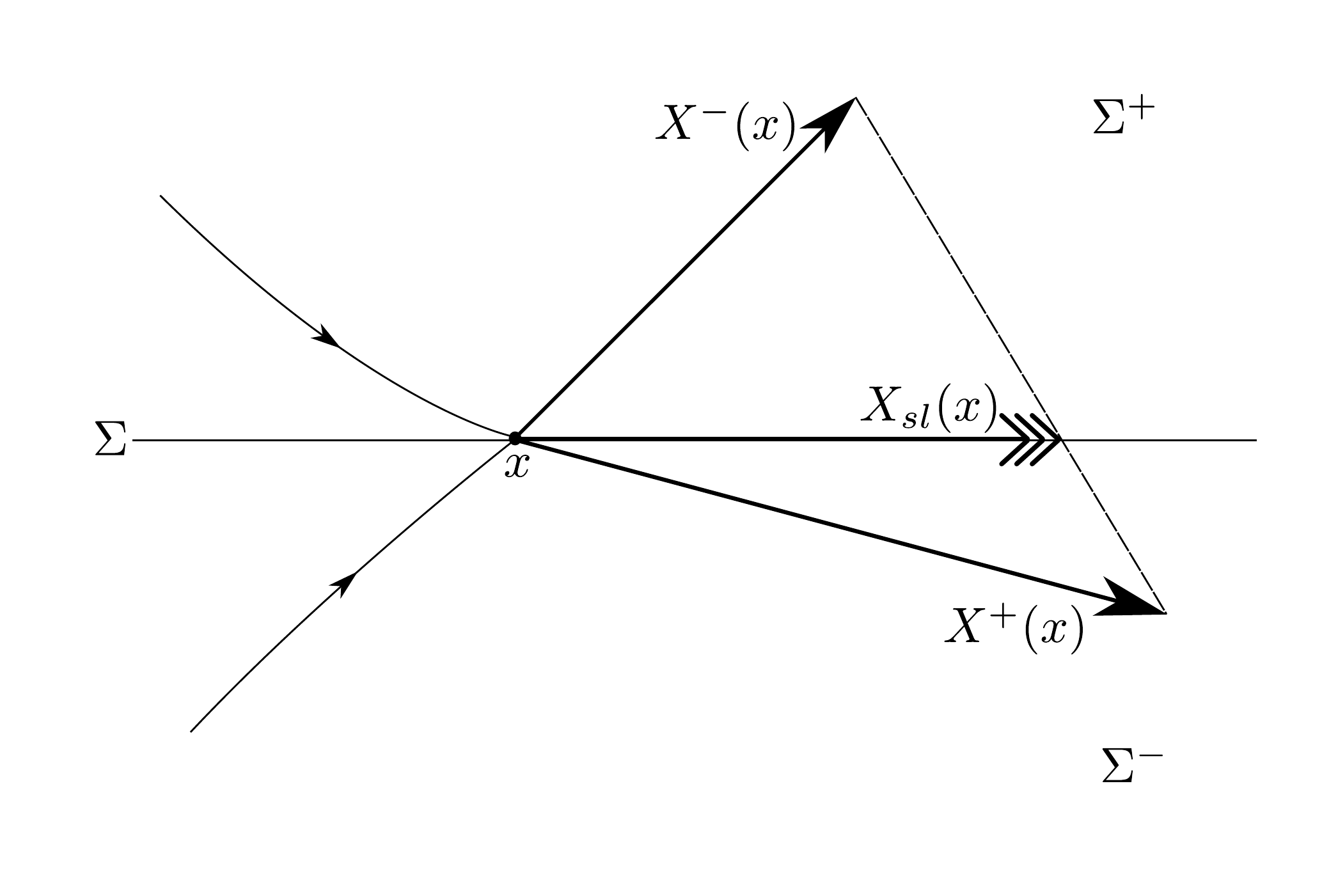}
 \caption{Illustration of the Filippov vector-field \eqref{eq.XSl}.}
\label{Filippov}
 \end{center}
              \end{figure}
              For \eqref{eq.yPos} and \eqref{eq.yNeg} we have $\lambda=\frac{1}{1-2x}$ and
\begin{align}
%
 X_{sl}:\quad \dot x=\frac{1}{1-2x},\,\dot y=0,\quad (x,y)\in \Sigma_{sl}.\label{eq.Filippov}
\end{align}
This gives the PWS phase portrait illustrated in Fig. \ref{GeneralSetupR2}. Notice in particular that all orbits that reach $\Sigma_{sl}$ leave $\Sigma$ at $T$ following $X_{sl}$ and $\{y=x^2\vert x>0\}$.

\subsection{Regularization: Slow-fast analysis}
In \cite{reves_regularization_2014} the authors regularize the PWS system $X=(X^+,X^-)$, described by \eqref{eq.yPos} and \eqref{eq.yNeg}, through a Sotomayor and Teixeira \cite{Sotomayor96} regularization:
\begin{align*}
 X_\varepsilon = \frac12 X^+(1+\phi(y\varepsilon^{-1}))+\frac12 X^- (1-\phi(y\varepsilon^{-1}).
\end{align*}
 \ed{From \eqref{eq.yPos} and \eqref{eq.yNeg} we obtain
\begin{align}
 x' &= \frac12 \varepsilon (1+\phi(y\varepsilon^{-1})),\label{eq.xyPWSFast}\\
 y'&= \varepsilon \left(x (1+\phi(y\varepsilon^{-1}))+\frac12 (1-\phi(y\varepsilon^{-1}))\right),\nonumber
\end{align}
in terms of the fast time $\tau = \varepsilon^{-1} t$. }
Here the function $\phi$ belongs to the following set $C_{ST}^k$ of functions:
 \begin{definition}\label{STphi}
The set $C_{ST}^k$ of Sotomayor and Teixeira regularization functions $\phi$ satisfy:
\begin{itemize}
 \item[$1^\circ$] \textnormal{Finite deformation}: \begin{eqnarray}
 \phi(y)=\left\{\begin{array}{cc}
                 1 & \text{for}\quad y\ge 1,\\
                 \in (-1,1)& \text{for}\quad y\in (-1,1),\\
                 -1 & \text{for}\quad y\le -1,\\
                \end{array}\right.\label{phiFunc}
 \end{eqnarray}
\item[$2^\circ$] \textnormal{Monotonicity}: \begin{align}
 \phi'(y)>0 \quad \text{within}\quad y\in (-1,1).\label{phiProperties}
 \end{align}
 \item[$3^\circ$] \textnormal{Finite $C^k$-smoothness}: $\phi\in C^\infty$ within $y\in (-1,1)$ but there exists a smallest $k\ge 1$ so that $\phi^{(k+1)}$ is discontinuous at $y=\pm 1$: $\phi^{(k+1)}(\pm 1^\mp)\ne 0$.
\end{itemize}
\end{definition}
An example of a $C_{ST}^1$ regularization function within this class is the following function
 \begin{eqnarray}
 \phi(y) = \left\{\begin{array}{ccc}-\frac12 y^3 + \frac32 y& \text{for}& y\in (-1,1),\\
                   \pm 1 & \text{for}&y \gtrless \pm 1.
                  \end{array}\right.
\label{phi1}
 \end{eqnarray}
 Here $\phi^{(1)}(\pm 1^\mp)=0$ but $\phi^{(2)}(\pm 1^\mp)=\mp 3$ (while $\phi^{(2)}(\pm 1^\pm)=0$) and hence $k=1$ in $3^\circ$ of Definition \ref{STphi} for this example. For simplicity, we restrict attention to $k\ge 1$ and therefore exclude consideration of the following $C^0$-function:
 \begin{align*}
  \phi(y) = \left\{\begin{array}{ccc} y & \text{for}& y\in (-1,1),\\
                   \pm 1 & \text{for}&y \gtrless \pm 1.
                  \end{array}\right.
 \end{align*}
 Note that $1^\circ$ also excludes analytic functions such as $\tanh$.

 \ed{Now, whereas \eqref{eq.xyPWSFast} is not slow-fast, setting
 \begin{align}
 \hat y=y\varepsilon^{-1},\label{eq.haty}
 \end{align}
gives a slow-fast system:
\begin{align}
 x'&=\frac{\varepsilon}{2} (1+\phi(\hat y)),\label{eq.xhatySlow}\\
 {\hat{y}}'&=x(1+\phi(\hat y))+\frac12 (1-\phi(\hat y)),\nonumber
\end{align}}
%
%
or
 \begin{align}
  \dot x &=\frac12 (1+\phi(\hat y)),\label{eq.xhaty}\\
 \varepsilon \dot{\hat{y}}&=x(1+\phi(\hat y))+\frac12 (1-\phi(\hat y)),\nonumber
\end{align}
in terms of the slow time $t = \varepsilon \tau$.
Hence the layer problem for this system is
\begin{align}
 x' &=0,\label{eq.layer}\\
 \hat y'&=x(1+\phi(\hat y))+\frac12 (1-\phi(\hat y)),\nonumber
\end{align}
while
\begin{align}
 \dot x &=\frac12 (1+\phi(\hat y)),\label{eq.reduced}\\
 0&=x(1+\phi(\hat y))+\frac12 (1-\phi(\hat y)),\nonumber
\end{align}
 is the reduced problem. 

The Filippov convention appears naturally in mechanics, but it also has the following desirable property:
\begin{theorem}\label{fenichel}
 \cite{krihog,llibre_sliding_2008} The $C_{ST}^k$-regularized system $X_\varepsilon$ possesses an attracting critical manifold $S_a$ for $\varepsilon=0$ which is a graph over $\Sigma_{sl}$. The reduced equations for the slow variable $x$ coincides with Filippov's sliding equations. The critical manifold is nonhyperbolic at $x=0,\,\hat y=1$.
\end{theorem}
\begin{proof}
 We demonstrate this result by considering our model system \eqref{eq.xhaty}. The critical manifold:
 \begin{align}
  \phi(\hat y) = \frac{1+2x}{1-2x},\label{eq.phiSol}
 \end{align}
 is only a graph over $\overline{\Sigma}_{sl}$; within $\Sigma_{cr}$ the expression on the right hand side becomes $\gtrless \pm 1$. By linearizing \eqref{eq.layer} about \eqref{eq.phiSol} we realise that the manifold \eqref{eq.phiSol} is attracting within $\Sigma_{sl}$ but nonhyperbolic at $x=0$ since $\phi'(1)=0$ there, cf. $3^\circ$ of Definition \ref{STphi}. 
 
 The reduced equations for the slow variable then becomes
 \begin{align}
  \dot x &=\frac12 (1+\phi(\hat y)) = \frac{1}{1-2x},\label{eq.reducedxDot}
 \end{align}
 using \eqref{eq.phiSol}. This expression coincides with \eqref{eq.Filippov}. 
\end{proof}
The proof does not use $1^\circ$ and $k<\infty$ in $3^\circ$ of Definition \ref{STphi}. This result therefore applies to a larger set of functions $\phi$, including analytic functions such as $\tanh$, satisfying:
\begin{align}
 \phi(\hat y)\in [-1,1],\,\, \phi'(\hat y) >0\,\, \text{for}\,\, \phi^{-1}(\hat y)\in (-1,1),\,\,\, \phi(\hat y) \rightarrow \pm 1\quad \text{for}\quad \hat y \rightarrow \pm \infty.\label{eq.regClass}
\end{align}

By Fenichel's theory, compact subsets of $S_a$ perturb to slow invariant manifolds $S_{a,\varepsilon}$ for $\varepsilon$ sufficiently small. The flow on $S_{a,\varepsilon}$ converges to the flow of the reduced problem for $\varepsilon\rightarrow 0$. The authors in \cite{reves_regularization_2014} investigate,  among other things, the intersection of $S_{a,\varepsilon}$ with a fixed section $\{y=\theta\}$ for $\phi\in C_{ST}^k$, $0\le k<\infty$. We find it easier to study the intersection of $S_{a,\varepsilon}$ with $\{x=\theta\}$ rather than $\{y=\theta\}$, but essentially the result of \cite[Theorem 2.2]{reves_regularization_2014} for $k\ge 1$ is the following:
\begin{theorem}\label{thm.Bonet}
 Consider $\phi\in C_{ST}^{n-1}$, $n\ge 2$. Let 
 \begin{align}
  r_2 = \varepsilon^{1/(2n-1)}.\label{r2Bonet}
 \end{align}
Then the slow manifold $S_{a,\varepsilon}$ intersects $\{x=\theta\}$ in $(\theta,y_\theta(\varepsilon))$ with
 \begin{align}
  y_\theta(\varepsilon) = \theta^2+\varepsilon-r_2^{2n} \left(\frac{2}{\phi^{[n]}}\right)^{2/(2n-1)} \eta(n)^2(1+r_2 F(r_2)),\label{eq.yTheta0}
 \end{align}
where 
\begin{itemize}
\item $F$ is smooth;
\item $\phi^{[n]}=\frac{(-1)^{n+1}}{n!}\phi^{(n)}(1)>0$;
\item $\eta(n)>0$ is a positive constant depending only on $n$. 
\end{itemize}
Also let $\rho,\,\nu>0$. Then the mapping $P_\varepsilon$ from $\{x=-\rho,\,y\in[-\nu, \nu]\}$ to $\{x=\theta\}$, obtained from the forward flow, satisfies $P_\varepsilon(y) = y_\theta(\varepsilon)+\mathcal O(e^{-c/\varepsilon})$, $P_\varepsilon'(y) = \mathcal O(e^{-c/\varepsilon})$. 
\end{theorem}

\begin{proof}
 This result is to order $\varepsilon$ obtained by setting $x_0^+=\sqrt{y}$, $\alpha^+=-1/\sqrt{y}$, in the expression for $P_\varepsilon(x)$ in the second point of the \ed{itemized list} in  \cite[Theorem 2.2]{reves_regularization_2014}, and solving $P_\varepsilon(x)=\theta$ for $y$. The authors in \cite{reves_regularization_2014} use asymptotic methods.  We will in Appendix \ref{appA} demonstrate an alternative proof using the blowup method which will give rise to the complete expression in \eqref{eq.yTheta0}. 
\end{proof}


\subsection{Geometry of regularization}
In this paper we shall extend Theorem \ref{thm.Bonet} to the description of the intersection of $S_{a,\varepsilon}$ with $\{x=\theta\}$ for the analytic regularization function $$\phi(\hat y) = \tanh(\hat y).$$ In terms of the application of slow-fast theory, this adds a significant amount of complexity since here $\phi'>0$ for all $\hat y\in \mathbb R$. This implies, in contrast to the case of $C_{ST}^k$-functions, that the critical manifold loses hyperbolicity at infinity $\hat y\rightarrow \infty$ (rather than at $\hat y=1$). To handle this, it is useful to consider the scaling $y=\varepsilon \hat y$ in \eqref{eq.haty} as part of a blowup:
\begin{align}
 y=\pi \bar y, \quad \varepsilon = \pi\bar \varepsilon,\quad \pi\ge 0,\,(\bar y,\bar\varepsilon)\in S^1,\label{eq.blowup0}
\end{align}
of $x\in \mathbb R,\,y=\varepsilon=0$.  
Then $y=\varepsilon \hat y$ becomes a scaling chart $$\bar\varepsilon=1:\quad y=\hat \pi\hat y,\,\,\varepsilon=\hat \pi.$$ In the $(x,\hat y)$-system the layer problem has $\dot x=0$ for $\varepsilon=0$. The PWS system is therefore not visible in this chart for $\varepsilon=0$. To connect to the PWS system, one can consider the two charts $$\bar y=\pm 1:\quad y=\pm \hat \pi,\,\,\varepsilon=\hat \pi\hat \varepsilon.$$ to cover $\bar y>0$ and $\bar y<0$, respectively, of the sphere $(\bar y,\bar \varepsilon)\in S^1$. See Fig. \ref{fig.New}.

\begin{figure}[h!t!]\begin{center}{\includegraphics[width=.65\textwidth]{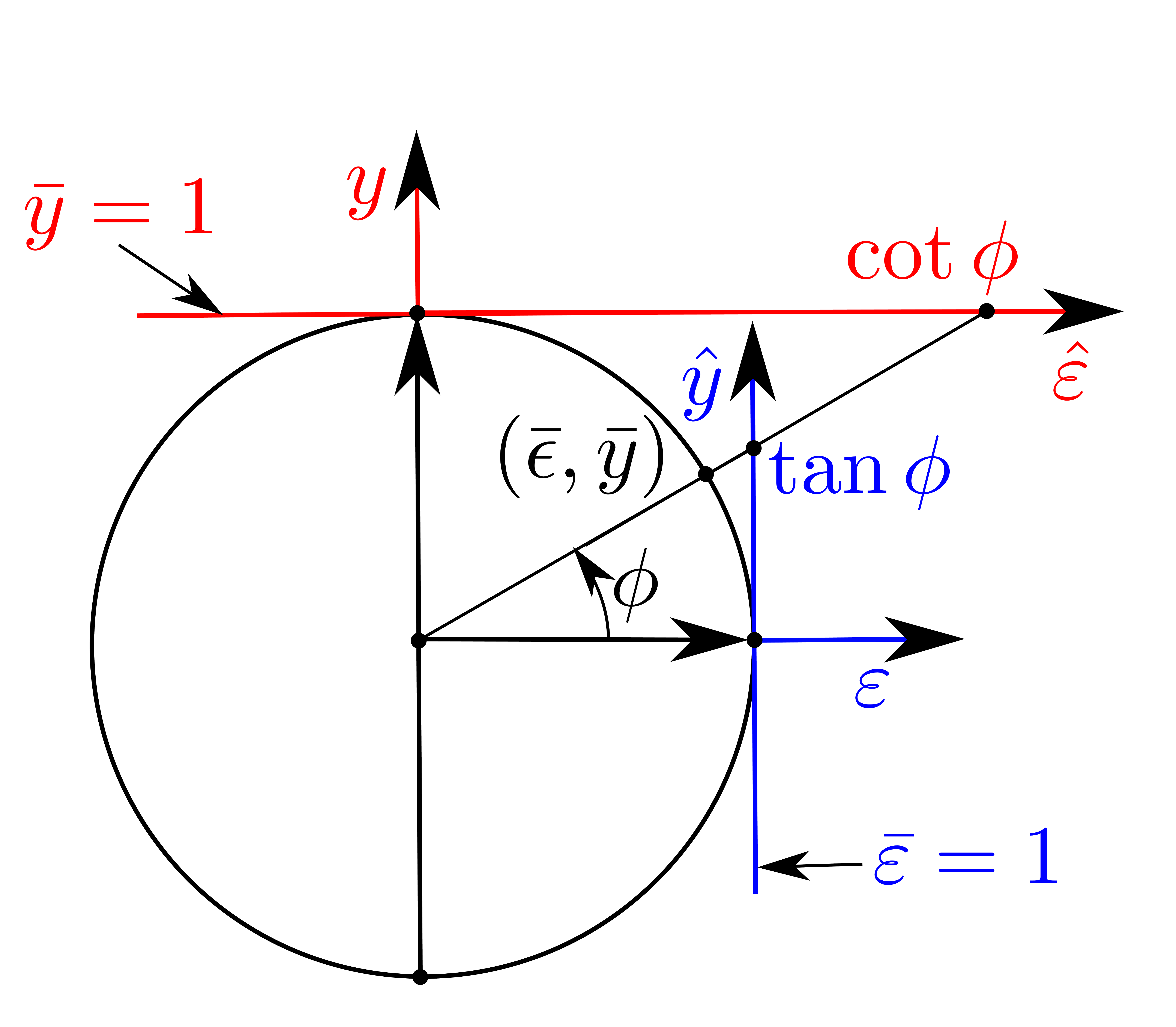}}
\caption{Illustration of the blowup \eqref{eq.blowup0} and the charts $\bar \epsilon=1$ and $\bar y=1$.}
\label{fig.New}
\end{center}
\end{figure}

The advantage of the $C_{ST}^k$-functions is that the charts $\bar y=\pm 1$ are not needed. Indeed, if $\phi\in C_{ST}^k$ then $\phi(\hat y)=\pm 1$ for $\hat y\gtrless \pm 1$ cf. $1^\circ$ in Definition \ref{STphi}, whence
\begin{align*}
 X_\varepsilon = X^{\pm} \quad \text{for}\quad y\gtrless \pm \varepsilon.
\end{align*}
 Therefore we can just \textit{scale back down} and return to $y$ (as it was done in \cite{krihog2}) whenever $\hat y\notin (-1,1)$ (corresponding to $y\notin (-\varepsilon,\varepsilon)$ using \eqref{eq.haty}). This then leads to the following interpretation of phase space in the case of $C_{ST}^k$-functions: We continue orbits of $X^\pm$ that reach $y=\pm \varepsilon$, respectively, within $\hat y\in (-1,1)$ using \eqref{eq.xhatySlow}. Once an orbit of \eqref{eq.xhatySlow} reaches $\hat y=\pm 1$ again then this orbit can be continued using the PWS vector-fields $X^{\pm}$ from $y=\pm \varepsilon$, respectively. This also leads to a (singular) description for $\varepsilon=0$. Geometrically, it corresponds to blowing up the plane $y=0$ to $\hat y\in [-1,1]$ for $\varepsilon=0$. See Fig. \ref{GeneralSetupR2Eps0}.
 \begin{figure}
\begin{center}
\includegraphics[width=.7\textwidth]{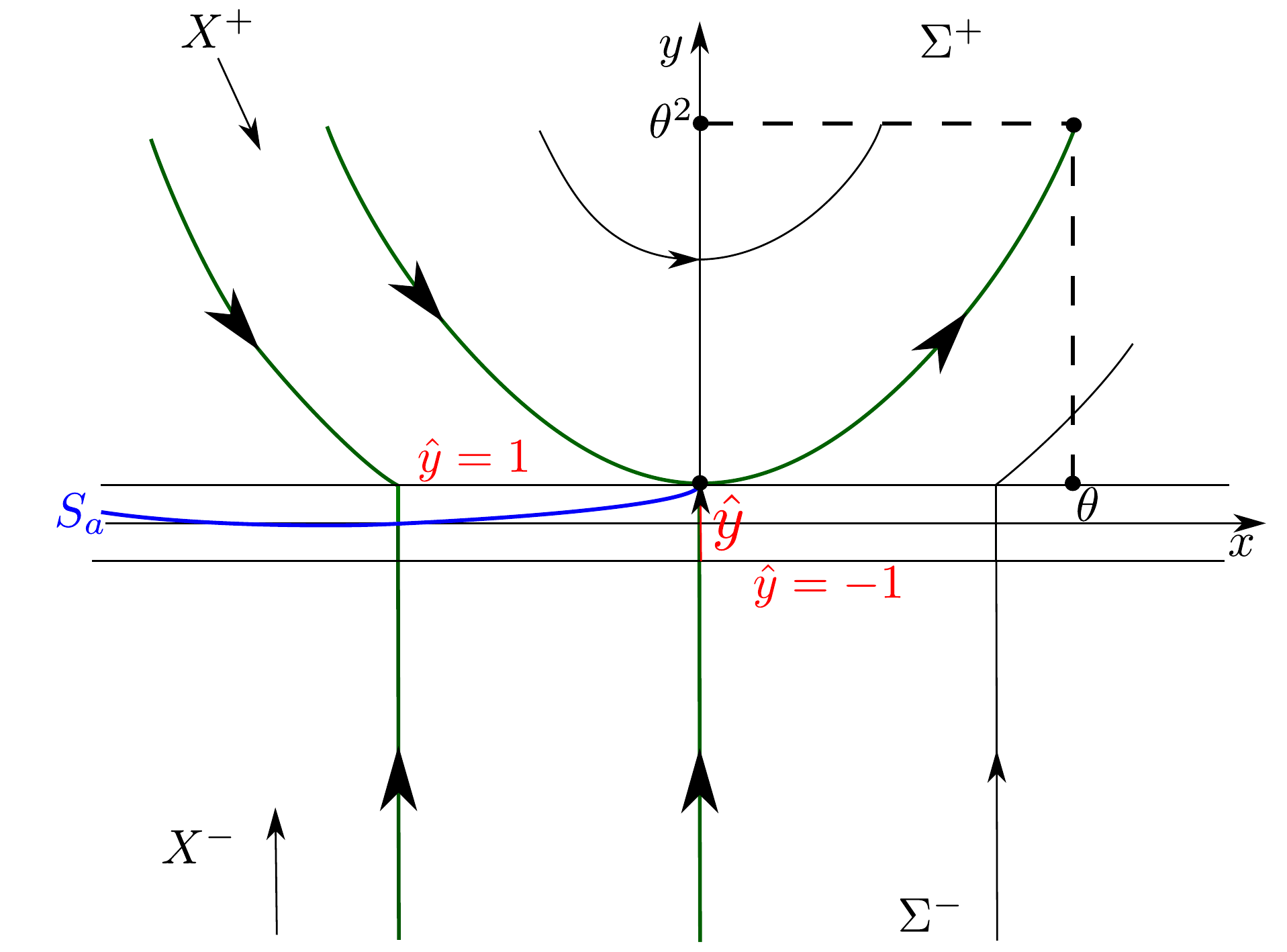}
 \caption{Singular slow-fast geometry for $\phi\in C_{ST}^k$.}
\label{GeneralSetupR2Eps0}
 \end{center}
              \end{figure}

\subsection{Regularization by $\tanh$}
For regularization functions such as $\tanh$ within the class 
\eqref{eq.regClass}, the compactification by $\hat y\in [-1,1]$ in Fig. \ref{GeneralSetupR2Eps0} is not possible and we need the charts $\bar y=\pm 1$ to connect $y=\mathcal O(\varepsilon)$ with $y=\mathcal O(1)$. 
But furthermore, which is relevant for our purposes, $\tanh(\hat y)$ is \textit{flat} for $\hat y\rightarrow \pm \infty$. Indeed, let $\hat y=\hat \varepsilon^{-1}$. Then
\begin{align}
\tanh(\hat \varepsilon^{-1})=1-\frac{2e^{-2\hat \varepsilon^{-1}}}{1+e^{-2\hat \varepsilon^{-1}}},\label{eq.asymptoticPhi}
\end{align}
with all derivatives at $\hat \varepsilon=0$ vanishing. \ed{We will in the following sections \ref{sec.bareps1} and \ref{sec.pwsChartYBarEq1} first demonstrate the use of the charts $\bar \varepsilon=1$ and $\bar y=1$. In particular in section \ref{sec.pwsChartYBarEq1}, where we describe the chart $\bar y=1$, we show that \eqref{eq.asymptoticPhi} gives rise to a flat critical manifold.} \ed{For simplicity we eliminate time through the division of $\frac12 (1+\phi(y\varepsilon^{-1}))$ and consider the following system
\begin{align}
 \dot x &=\varepsilon,\label{eq.tanhNewxy}\\
 \dot y &=\varepsilon \left(2x+\frac{1-\tanh(y\varepsilon^{-1})}{1+\tanh(y \varepsilon^{-1})}\right).\nonumber
\end{align}}

\subsection{Chart $\bar\varepsilon=1$}\label{sec.bareps1}
Inserting \eqref{eq.haty} into \eqref{eq.tanhNewxy} gives the equations 
\begin{align}
 \dot x&= \varepsilon ,\label{eq.tanhNew}\\
 \dot{\hat y}&=2x+\frac{1-\tanh(\hat y)}{1+\tanh(\hat y)} = 2x + e^{-2\hat y}.\nonumber
\end{align}
In accordance with Theorem \ref{fenichel}, this system possesses an attracting critical manifold $S_a$ as a graph over $\Sigma_{sl}$. By Fenichel's theory compact subsets of this manifold perturbs to an invariant slow manifold $S_{a,\varepsilon}$ for $\varepsilon$ sufficiently small. A simple computation shows the following:
\begin{lemma}\label{lemmaSaEpsTanh}
For $0<\varepsilon\ll 1$ the attracting slow manifold $S_{a,\varepsilon}$ intersects $\{\hat y=\xi^{-1}\}$, with $\xi$ small but fixed, in $(x_\xi(\varepsilon),\xi^{-1})$ with
\begin{align}
 x_\xi(\varepsilon)&=-\frac12 e^{-2\xi^{-1}} + \frac{\varepsilon}{2} e^{2\xi^{-1}}+\mathcal O(\varepsilon^2).\label{eq.xdelta}
\end{align}
\end{lemma}
\begin{proof}
 The critical manifold intersects $\{y=\xi^{-1}\}$ in $(x_\xi(0),\xi^{-1})$. Since $S_{a,\varepsilon}$ is $\mathcal O(\varepsilon)$-smoothly-close to $S_a$ the result follows from a simple computation. 
\end{proof}
Clearly $x$ increases (and therefore also $\hat y$) on $S_{a,\varepsilon}$ for $0<\varepsilon\ll  1$. To continue $S_{a,\varepsilon}$ near $\hat y=\infty$ we consider the chart $\bar y=1$.
\subsection{Chart $\bar y=1$}\label{sec.pwsChartYBarEq1}
This chart corresponds to setting $\bar y=1$ in \eqref{eq.blowup0}:
\begin{align*}
 y=\hat \pi,\quad \varepsilon = \hat \pi \hat \varepsilon,
\end{align*}
or simply:
\begin{align}
 \hat \varepsilon = y^{-1} \varepsilon=\hat y^{-1}.\label{eq.hatepsNew}
\end{align}
Inserting this into \eqref{eq.tanhNewxy} gives the following system of equations
\begin{align}
 \dot x &= \varepsilon=\hat \varepsilon y,\label{eq.baryEqn}\\
 \dot{ y} &= \varepsilon \left(2x+e^{-2\hat \varepsilon^{-1}}\right)=\hat \varepsilon y\left(2x+e^{-2\hat \varepsilon^{-1}}\right),\nonumber\\
 \dot{\hat \varepsilon} &=-y^{-2} \varepsilon \dot y =-\hat \varepsilon^2 \left(2x+e^{-2\hat \varepsilon^{-1}}\right).\nonumber
\end{align}
\ed{(We avoid the natural desingularization of $\hat \varepsilon=0$ through division by $\hat \varepsilon$ on the right hand side, since we would have to undo it, once we apply our method \mthref{method} in section \ref{sec.proofPWS}, for an extended system to be well-defined at $\hat \varepsilon=0$.)}
Setting $y=0$ in \eqref{eq.baryEqn} gives a new layer problem:
\begin{align}
 \dot x &= 0,\label{eq.layerbary}\\
 \dot{ y} &= 0,\nonumber\\
 \dot{\hat \varepsilon} &=-\hat \varepsilon^2 \left(2x+e^{-2\hat \varepsilon^{-1}}\right),\nonumber
\end{align}
for which the critical manifold $S_a$ from the chart $\bar \varepsilon=1$ becomes
\begin{align}
 x = -\frac12 e^{-2\hat \varepsilon^{-1}},\,y=0,\,\hat \varepsilon>0;\label{eq.criticalbary}
\end{align}
again a set of fix-points of \eqref{eq.layerbary}.
%
In agreement with the analysis in chart $\bar\varepsilon=1$ this set is normally attracting for $x<0$ but since \eqref{eq.criticalbary} is flat as a graph over $\hat \varepsilon\ge 0$ it loses hyperbolicity at $(x,y,\hat \varepsilon)=(0,0,0)$ at an exponential rate as $\hat \varepsilon\rightarrow 0^+$. Indeed, the linearization of \eqref{eq.layerbary} about \eqref{eq.criticalbary} gives
\begin{align}
-{2e^{-2\hat \varepsilon^{-1}}},\label{pwsEigenvalue}
\end{align}
as a single nontrivial eigenvalue. We will therefore apply \mthref{method} to this problem and, as in Theorem \ref{thm.Bonet}, describe the intersection of the forward flow of $S_{a,\varepsilon}$ with $\{x=\theta\}$.


\subsection{Main result}
Using the geometric method developed in this paper we prove the following:
\begin{theorem}\label{thm.tanh}
 Consider the analytic regularization function $\phi(\hat y)=\tanh (\hat y)$. Then the slow manifold of \eqref{eq.xhaty} intersects $\{x=\theta\}$ in $(\theta,y_\theta(\varepsilon))$ with
 \begin{align*}
  y_\theta(\varepsilon) = \theta^2 + \varepsilon \left(\frac14 \ln \left(\frac{\pi}{2}\varepsilon^{-1}\right)+R(\sqrt{\varepsilon})\right),
 \end{align*}
for some smooth function $R(\sqrt{\varepsilon})=\mathcal O(e^{-c\varepsilon^{-1}})$. 

Also let $\rho,\,\nu>0$. Then the mapping $P_\varepsilon$ from $\{x=-\rho,\,y\in[-\nu, \nu]\}$ to $\{x=\theta\}$, obtained from the forward flow, satisfies $P_\varepsilon(y) = y_\theta(\varepsilon)+\mathcal O(e^{-c/\varepsilon})$, $P_\varepsilon'(y) = \mathcal O(e^{-c/\varepsilon})$. 
\end{theorem}

\begin{remark}
We notice that the leading order correction in Theorem \ref{thm.tanh} for $y_\theta(\varepsilon)$ is $\mathcal O( \varepsilon \ln \varepsilon^{-1})$ while the corresponding expression  in Theorem \ref{thm.Bonet}, describing the regularization by $C_{ST}^{n-1}$-functions, is $\mathcal O(\varepsilon)$ (since $r_2^{2n}\ll \varepsilon$ cf. \eqref{r2Bonet}). Furthermore, we notice that the expression in Theorem \ref{thm.Bonet} is a smooth function of $\varepsilon$ and $r_2=\varepsilon^{1/(2n-1)}$. In comparison, the expression in Theorem \ref{thm.tanh} is only smooth as function of $\varepsilon,\,\ln \varepsilon^{-1}$ and $\sqrt{\varepsilon}$. Our approach identifies the origin of these terms.
\end{remark}

\begin{remark}
It is actually possible to integrate the system \eqref{eq.tanhNewxy}  with $\phi(\hat y)=\tanh (\hat y)$ directly. In Appendix \ref{direct} we show that the result in Theorem \ref{thm.tanh} is in agreement with direct integration. This example therefore provides a useful forum in which to
introduce our geometric approach. Needless to say, our method, relying only on hyperbolic methods and normal form theory, applies to regularization by $\tanh$ of my complicated systems, such as nonlinear versions of $X^{\pm}$ and PWS systems in higher dimensions.
\end{remark}




\subsection{Proof of Theorem \ref{thm.tanh}}\label{sec.proofPWS}
To deal with the loss of hyperbolicity of \eqref{eq.criticalbary} we first proceed as in step 1 of \mthref{method} and extend the phase space dimension by introducing 
\begin{align}
 q =  e^{-2\hat \varepsilon^{-1}},\label{eq.qq}
\end{align}
as a new dynamic variable. (Here $\hat \varepsilon$ plays the role of $y$ in \mthref{method}.) Following steps 2 and 3 in \mthref{method} we obtain
\begin{align*}
 \dot q &= {2e^{-2\hat \varepsilon^{-1}}}\hat \varepsilon^{-2} \dot{\hat \varepsilon}=2q\hat \varepsilon^{-2} \dot{\hat \varepsilon}.
\end{align*}
by differentiation of \eqref{eq.qq}. We therefore consider the extended system:
\begin{align}
 \dot x &= \varepsilon,\label{eq.extQ}\\
 \dot{\hat \varepsilon} &=-\hat \varepsilon^2 (2x+q),\nonumber\\
 \dot q &=-2q(2x+q),\nonumber\\
 \dot \varepsilon &=0,\nonumber
\end{align}
on the phase space: $$(x,\hat \varepsilon,q,\varepsilon)\in \mathbb R\times \overline{\mathbb R}_+^3.$$
\ed{(We will use \eqref{eq.hatepsNew} as
\begin{align}
 y =  \hat \varepsilon^{-1}\varepsilon,\label{eq.hallo}
\end{align}
whenever we wish to describe $y$ (see also Remark \ref{rem.pwsRemark} below).)}
The set \eqref{eq.criticalbary} then becomes 
\begin{align}
 S_a=\left\{(x,\hat \varepsilon,q,\varepsilon)\vert x=-\frac12 q,\,\,\hat \varepsilon>0,\,q>0,\varepsilon=0\right\},\label{eq.Saq}
\end{align}
abusing notation slightly.
Furthermore, we note that the set 
\begin{align}
 Q:\quad &q=e^{-2\hat \varepsilon^{-1}},\label{eq.setQTanh}
\end{align}
obtained from \eqref{eq.qq} is an invariant sets of \eqref{eq.extQ}. However, it is implicit in \eqref{eq.extQ} and we will evoke the invariance only when needed. 

\subsection{Blowup}
The line 
\begin{align}
\hat \varepsilon\ge 0,\,\,x=0,\,q=0,\varepsilon=0,\label{eq.NHline}
\end{align}
is a set of nonhyperbolic critical points of \eqref{eq.extQ}. 

We therefore consider the following blowup:
\begin{align}
 x = \bar r\bar x,\,\varepsilon = \bar r^2\bar \epsilon,\, q = \bar r\bar q,\label{eq.blowupTanh}
\end{align}
leaving (cf. step 4 in \mthref{method}) $\hat \varepsilon$ untouched.
The blowup transformation \eqref{eq.blowupTanh} gives rise to a vector-field $\overline X$ on $$(\hat \varepsilon,\bar r,(\bar x,\bar \epsilon,\bar q))\in \overline{\mathbb R}_+^2\times S^2.$$
Here $\overline X\vert_{\bar r=0}=0$ but the weights of $r$ in the blowup in \eqref{eq.blowupTanh} are chosen so that $\widehat X\equiv \bar r^{-1}\overline X\vert_{\bar r=0}$ is non-trivial. It is $\widehat X$ that we shall study in the following. We do so by considering the following charts:
\begin{align}
 \kappa_1:&\quad \bar q=1:\quad \quad x=r_1x_1,\,\varepsilon = r_1^2\epsilon_1,\,q = r_1,\label{eq.kappa1tanh}\\
 \kappa_2:&\quad \bar \epsilon=1:\quad \quad x=r_2x_2,\,\varepsilon = r_2^2,\, q = r_2q_2,\label{eq.kappa2tanh}
\end{align}
and
\begin{align}
 \kappa_3:&\quad \bar x=1:\quad \quad x=r_3,\,\varepsilon = r_3^2\epsilon_3,\,q = r_3q_3.\label{eq.kappa3tanh}
\end{align} 
Since $\hat \varepsilon$ is not transformed by this blowup transformation, we keep (as promised) using this symbols in the different charts. Geometrically, the line of critical points \eqref{eq.NHline} is upon \eqref{eq.blowupTanh} blown up to a cylinder of spheres: $\overline{\mathbb R}_+ \times S^2$.

\begin{remark}\label{rem.pwsRemark}
\ed{One could also write (\ref{eq.extQ}) as
\begin{align}
 \dot x &= \hat \varepsilon y,\label{eq.extQNew}\\
 \dot{ y} &= \hat \varepsilon y (2x+q),\nonumber\\
 \dot{\hat \varepsilon} &=-\hat \varepsilon^2 (2x+q),\nonumber\\
 \dot q &=-2q(2x+q),\nonumber
\end{align}
using \eqref{eq.hatepsNew}, 
and consider the phase space $$(x,y,\hat \varepsilon,q)\in \mathbb R\times \overline{\mathbb R}_+.$$ (In fact this would be useful if the piecewise smooth vector-fields $X^\pm$ were dependent upon $y$.) Then the conservation of the small parameter $\varepsilon$ would be {implicit} in \eqref{eq.extQNew} as an invariant foliation
\begin{align}
 \varepsilon = \hat \varepsilon y,\label{eq.Eset}
\end{align}
by $\varepsilon=\text{const}.\ge 0$. 
In the $(x,y,\hat \varepsilon,q)$-space, the nonhyperbolic line \eqref{eq.NHline} would become 
\begin{align*}
\hat \varepsilon\ge 0,\,\,x=0,y=0,\,q=0,
\end{align*}
which could be blown up as
\begin{align*}
 x = \bar r\bar x,\,y = \bar r^2\bar y,\, q = \bar r\bar q,\quad (\bar r,(\bar x,\bar y,\bar q))\in \overline{\mathbb R}_+\times S^2,
\end{align*}
in a similar fashion to \eqref{eq.blowupTanh}. However, for the purpose of demonstrating \mthref{method}, I prefer the formulation in (\ref{eq.extQ}) because it has (as the other examples considered in this paper) an explicit slow-fast structure. This streamlines the analysis with e.g. \cite{krupa_extending_2001,krupa_extending2_2001}. 
}
\end{remark}


\ed{In Appendix \ref{appC} we analyze each of the different charts $\kappa_1$, $\kappa_2$ and $\kappa_3$. The geometry is sketched in Fig. \ref{tanh}.
We summarise the results here: In the entry chart $\kappa_1$ \eqref{eq.kappa1tanh} we gain hyperbolicity of $S_a$ (blue in Fig.  \ref{tanh}) locally on the blowup cylinder $\bar r=0$: $(\hat \varepsilon,(\bar x,\bar y,\bar q))\in \overline{\mathbb R}_+\times S^2$. This provides, using the invariance of $Q$ \eqref{eq.setQTanh} and \eqref{eq.Eset},  (see Lemma \ref{lem.Ma1}, Remark \ref{rem.Ma1} and Fig. \ref{tanhkappa1}) an extension of Fenichel's slow manifold $S_{a,\varepsilon}$ (red in Fig. \ref{tanh}) $\mathcal O(\ln^{-1} \varepsilon^{-1})$-close to $\bar C_a$ (see \eqref{eq.Ca1}); a unique local center manifold within $(\hat \varepsilon,(\bar x,\bar y,\bar q))\in \{0\}\times S^2$ (green in Fig. \ref{tanh}). Then in the scaling chart $\kappa_2$ \eqref{eq.kappa2tanh} we carry the extension of the slow manifold across the blowup sphere from an explicit solution that guides $\bar C$ forward (see \eqref{eq.m2Eqn} and Lemma \ref{lem.Ma2}). This part is illustrated in  Fig. \ref{tanhkappa2}. Finally, in the exit chart $\kappa_3$ \eqref{eq.kappa3tanh} we find an inflowing, attracting center manifold $\bar N\subset \{\bar q=0\}$ (see \eqref{eq.N3}). The forward flow of $\bar C$ is contained within $W_{loc}^s(\bar N)$. Therefore, by applying Fenichel's normal form \cite{jones_1995} in Lemma \ref{lem.FNF}, we are finally in Lemma \ref{lem.P3} able to guide the slow manifold, through a set of reduced equations (see \eqref{eq.reduced3}) on the stable fibers of $\bar N$, up until the intersection with $\{x = \theta\}$ (see Fig. \ref{tanhkappa3}). This completes the proof of Theorem \ref{thm.tanh}. }

\begin{figure}
\begin{center}
\includegraphics[width=.9\textwidth]{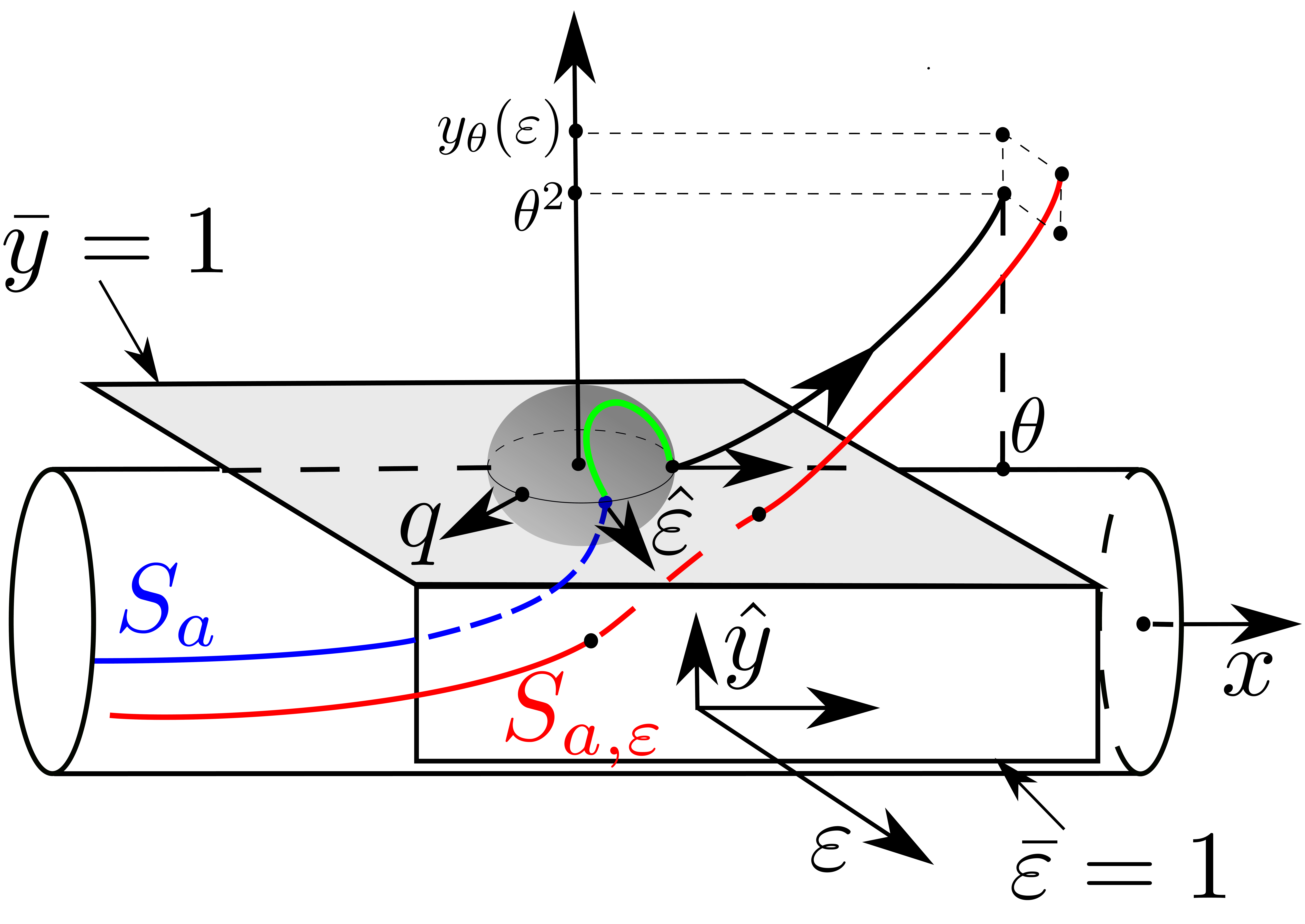}
 \caption{Illustration of the geometry associated to the analysis of $X_\varepsilon$. }
\label{tanh}
 \end{center}
              \end{figure}

%

\section{A flat slow manifold in a model for aircraft ground dynamics}\label{sec.aircraft}
 In this section we consider the system \eqref{eq.aircraftModel00}, repeated here for convinience:
 \begin{align}
 \dot u&= -\varepsilon(\alpha-v),\label{eq.aircraftModel01}\\
 \dot v&= -u-(v-a)e^{vb},\nonumber
\end{align}
  with parameters $a$ and $b$ fixed and use $\alpha$ as a bifurcation parameter. In comparison with \cite{aircraft} we have replaced $b$ by $b^{-1}$. The system \eqref{eq.aircraftModel01} possesses a critical manifold which is a graph over the fast variable $v$:
\begin{align*}
C=\{(u,v)\vert u= (a-v)e^{vb}\}.
\end{align*}
Linearization of \eqref{eq.aircraftModel01} about $C$ for $\varepsilon=0$ gives
\begin{align}
 -\left(b(v-a)-1\right) e^{vb},\label{eq.hyperbolicity}
\end{align}
as a single non-trivial eigenvalue. The manifold $C$ therefore splits into an attracting critical manifold:\footnote{Here we do not use $S_a$ and $S_r$ for attracting and repelling critical manifolds because we find that this becomes confusing when we later reverse the direction of time.}
\begin{align*}
 \hat S=C\cap \{v>a-b^{-1}\},
\end{align*}
a fold point:
\begin{align*}
 F = (u_f,v_f),
\end{align*}
with 
\begin{align*}
 u_f = b^{-1} e^{(a-b^{-1})b},\quad v_f=a-b^{-1},
\end{align*}
and a repelling critical manifold:
\begin{align*}
 S=C\cap \{v<a-b^{-1}\},
\end{align*}
Compact subsets of $\hat S$ and $S$ perturb by Fenichel's theory to $\hat S_\varepsilon$ and $S_\varepsilon$. 
The system undergoes a Hopf bifurcation near the fold $F$ at a parameter value $$\alpha = \alpha_H(\varepsilon)\equiv a-b^{-1}+\mathcal O(\varepsilon).$$ This leads to a canard explosion phenomenon near the canard value $\alpha_c\equiv a-b^{-1}+\mathcal O(\varepsilon)$ where the limit cycles born in the Hopf bifurcation undergoes $\mathcal O(1)$-changes within a parameter regime of width $\mathcal O(e^{-c/\varepsilon})$, $c>0$. See \cite{krupa_relaxation_2001}. Examples of limit cycles computed using AUTO by the present author are shown in Fig. \ref{aircraftNew} near $\alpha_c \approx -10^{-3}$ for $a=b=1$ and $\varepsilon=10^{-3}$. The small limit cycles, following the repelling manifold $S_\varepsilon$ before jumping directly towards the attracting manifold $\hat S_\varepsilon$, are frequently called canard cycles \textit{without head}. Similarly, the canard cycles that leave $S_\varepsilon$ on the other side, escaping towards infinity $v\rightarrow -\infty$ before returning to $\hat S_\varepsilon$, are said to be \textit{with head} (indicated by the eye). Eventually the canard cycles become relaxation oscillations that escape directly towards infinity $v\rightarrow -\infty$ at the fold $F$. Note how trajectories cross $S$ for $v\ll 0$. This is due to the fact that $S$ loses hyperbolicity at infinity $v\rightarrow -\infty$. Similar transition to relaxation oscillations occur in \cite[Fig. 4]{bro2} and \cite{Gucwa2009783}, but in \eqref{eq.aircraftModel01} the loss of hyperbolicity occurs at an exponential rate $\mathcal O(e^{vb})$. Therefore the classical theory of canard explosion \cite{krupa_relaxation_2001} does not describe the transition from canard cycles without head to those with head. The aim of this section is to apply our approach to \eqref{eq.aircraftModel01} and obtain a description of this transition as $\varepsilon\rightarrow 0$ and establish the existence of large canard cycles. 

\begin{figure}
\begin{center}
\includegraphics[width=.7\textwidth]{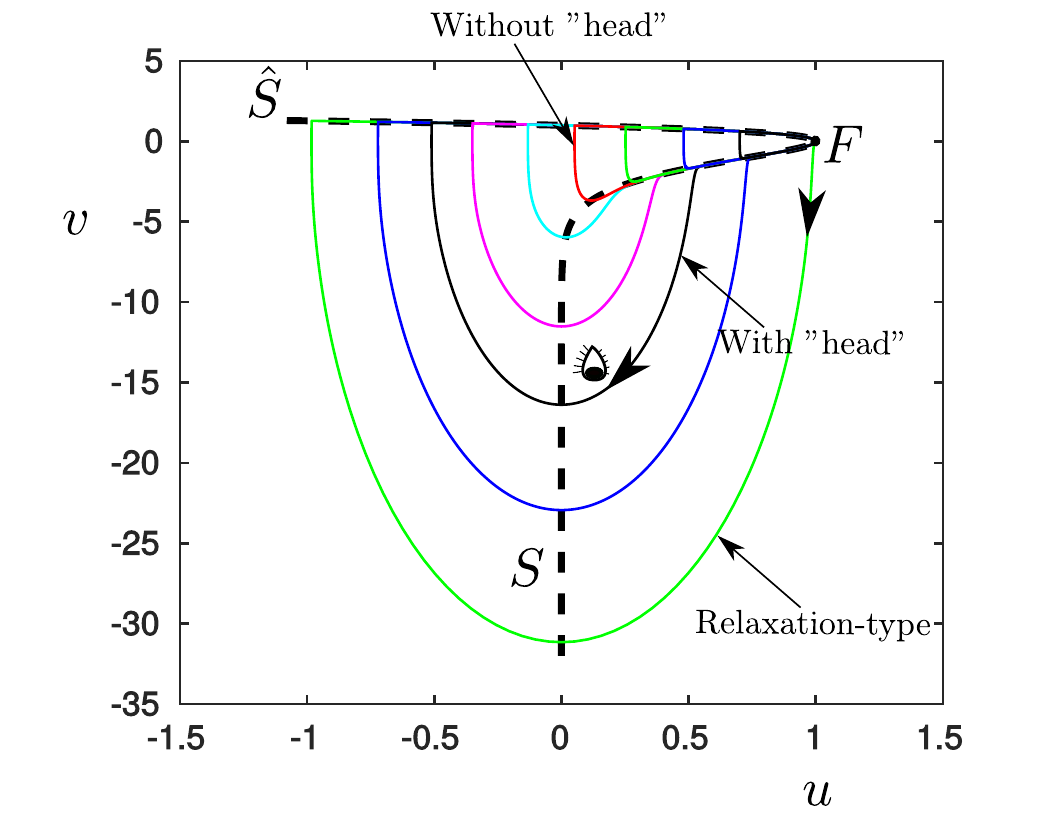}
 \caption{Different canard cycles: without head (e.g. the red orbit), with head (e.g. the large black orbit with the eye) and relaxation-type oscillations (large green orbit), near the canard explosion for $a=b=1$, $\varepsilon=10^{-3}$ and $\alpha \approx -10^{-3}$.}
\label{aircraftNew}
 \end{center}
              \end{figure}

\subsection{Setup}
To obtain large amplitude limit cycles near the canard value, we proceed as in the classical analysis \cite{krupa_relaxation_2001}, and consider two sections at $\{v=v_f\}$:
\begin{align}
\Gamma &= \{(u,v)\vert u\in [-c^{-1},c^{-1}],\,v=v_f\},\label{eq.GammaNeg}\\
\Lambda &=\{(u,v)\vert u\in [-c^{-1}+u_f,c^{-1}+u_f],\,v=v_f\},\label{eq.Lambda}
\end{align}
for $c$ sufficiently small. See Fig. \ref{Caircraft}. 
\begin{figure}
\begin{center}
\includegraphics[width=.6\textwidth]{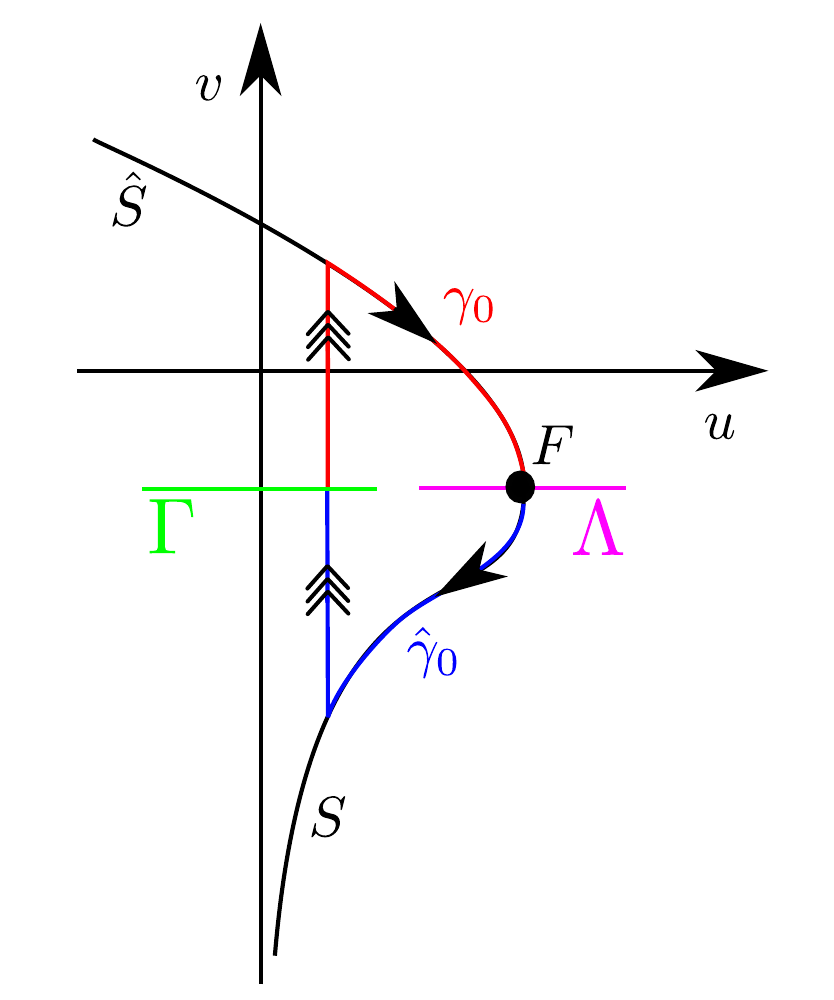}
 \caption{Illustration of the critical manifold $C=\hat S\cup F\cup S$ and the two sections $\Gamma$ and $\Lambda$.}
\label{Caircraft}
 \end{center}
              \end{figure}
Let $(u,v)\in \Gamma$ and consider the forward orbit $\gamma_\varepsilon(u,\alpha)$ and backward orbit $\hat \gamma_\varepsilon(u,\alpha)$. Denote their first intersection with $\Lambda$ by $d_\varepsilon(u,\alpha)$ and $\hat d_\varepsilon(u,\alpha)$, respectively. For 
\begin{align}
u\ge \bar c^{-1}>0,\label{eq.uBarC}
\end{align}
with $\bar c>c$ fixed but large, the standard theory applies. Indeed, by the transverse intersection of (fixed copies of) $\hat S_\varepsilon$ and $S_\varepsilon$ at $\alpha=\alpha_c$ it is possible to solve $d_\varepsilon(u,\alpha)=\hat d_\varepsilon(u,\alpha)$ for $\alpha=\alpha_c+\mathcal O(e^{-c/\varepsilon})$ by the implicit function theorem. This argument fails near $u=0$ due to the loss of hyperbolicity of $S$ at $u=0$. Geometrically, the loss of hyperbolicity, as in the models \eqref{eq.kuehn0} considered by Kuehn in \cite{kuehn2014}, is due to the asymptotic alignment for $v\rightarrow -\infty$ of the tangent spaces $TS$ with the critical fibers. In the following we will study $S$ using our methods and describe the limit cycles that intersect $\Gamma$ with $u\approx 0$. For this we reverse time so that $S$ becomes attracting and consider the equations:
\begin{align}
 \dot u &= \varepsilon(\alpha-v),\label{eq.aircraftModel0}\\
 \dot v&= u+(v-a)e^{vb}.\nonumber
\end{align}
\subsection{Equations at infinity}
To deal with the loss of hyperbolicity for $v\rightarrow -\infty$ we introduce the following \textit{chart}:
\begin{align}
 x = -uv^{-1},\quad y = -v^{-1},\label{eq.xyFromWz}
\end{align}
based upon Poincar\'e compactification \cite{chicone}. Then $y=0^+$ corresponds to $v=-\infty$. It is convenient (rather than a necessity) if the fold $F$ is visible in the chart \eqref{eq.xyFromWz} and we shall therefore henceforth assume that $$v_f=a-b^{-1}<0.$$  This allows us to only work in the $(x,y)$-variables. 
Using \eqref{eq.aircraftModel0} and \eqref{eq.xyFromWz} we obtain the following equations:
\begin{align}
 \dot x &=y\left(\varepsilon(1+\alpha y)+x\left(x-(1+ay)e^{-y^{-1}b}\right)\right),\label{eq.aircraftInfty}\\
 \dot y&=y^2\left(x-(1+ay)e^{-y^{-1}b}\right),\nonumber
\end{align}
after multiplication of the right hand side by $y$. This corresponds to a nonlinear transformation of time for $y>0$ (which will become useful later on in \eqref{eq.aircraftQ}, recall the remark in step 3 of \mthref{method}). 
The critical manifold $S$ from above becomes
\begin{align*}
 S:\quad x = (1+ay)e^{-y^{-1}b},\quad y\in (0,y_f)
\end{align*}
in the $(x,y)$-variables while the fold $F$ becomes
\begin{align*}
 F=(x_f,y_f).
\end{align*}
Here $x_f = -u_fv_f^{-1}$, $y_f=-v_f^{-1}$. The manifold $\hat S$ is only partially visible ($v<0$ only) in the $(x,y)$-variables. 
We will continue to denote $\hat S$, $S$ and $F$ by the same symbols in the new variables $(x,y)$. The set $S$ is still a set of critical points of \eqref{eq.aircraftInfty}$_{\varepsilon=0}$. But as opposed to \eqref{eq.aircraftModel0}, where the fibers were vertical: $\mathcal F_{u_0}=\{(u_0,v)\vert v\in \mathbb R\}$ for $\varepsilon=0$, the singular fibers are now tilted:
\begin{align*}
 \mathcal F_{u_0}=\{(x,y)\vert x=u_0y\},
\end{align*}
with respect to the $(x,y)$-variables. See also Fig. \ref{aircraftNewNew} where segments of the limit cycles in Fig. \ref{aircraftNew} are illustrated in the $(x,y)$-variables. 

\begin{figure}
\begin{center}
\includegraphics[width=.7\textwidth]{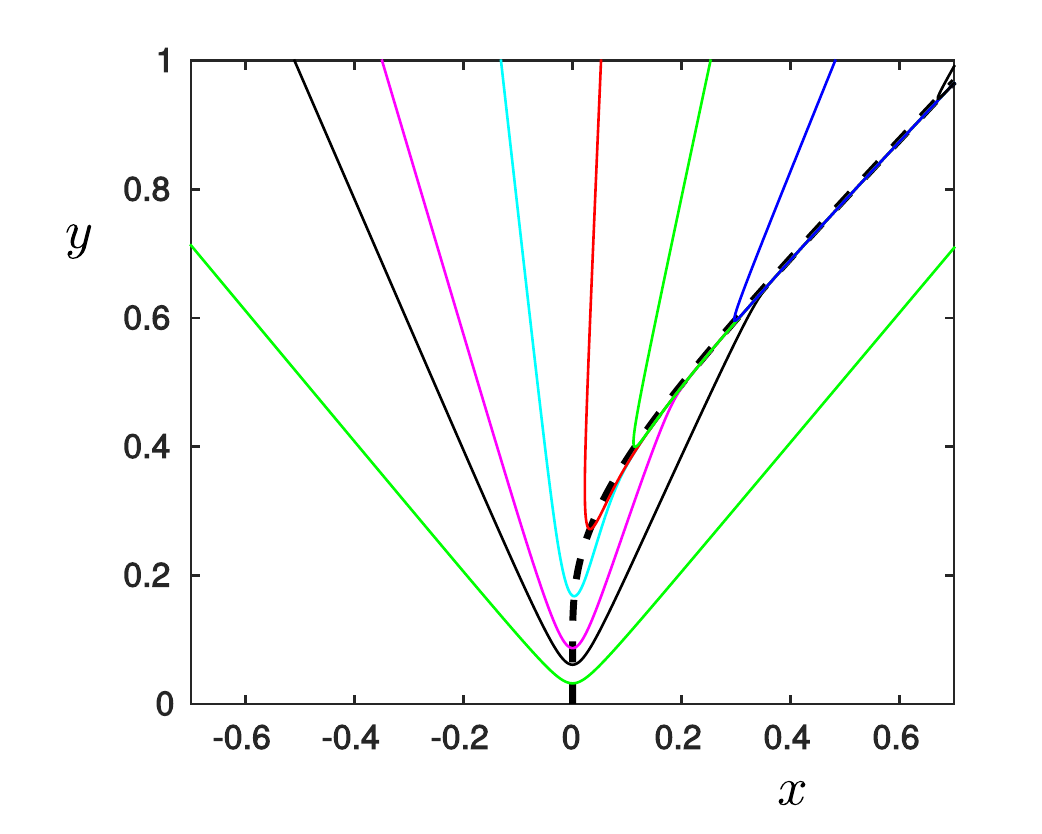}
 \caption{The limit cycles from Fig. \ref{aircraftNew} in the $(x,y)$-variables.}
 \label{aircraftNewNew}
\end{center}
              \end{figure}

The fibers to the left of $S$ all emanate from the unstable fibers of $\hat S$. Note again that $S$ is flat at $y=0$ and therefore loses hyperbolicity at $y=0$ at an exponentially rate. Indeed, the linearization of \eqref{eq.aircraftInfty}$_{\varepsilon=0}$ about $S$ yields
\begin{align*}
 -be^{-y^{-1}b}(1+(a-b^{-1})y),
\end{align*}
as a single non-trivial eigenvalue. 

The section $\Gamma$ in \eqref{eq.GammaNeg} becomes
\begin{align*}
 \Gamma=\{(x,y)\vert x\in [-\nu,\nu],\,y=y_f\},
\end{align*}
with $\nu = -c^{-1}v_f^{-1}$,
in the new variables. We again use the same symbol for this section. 

\subsection{Main result}
By \cite{krupa_relaxation_2001}, a fixed slow manifold $S_\varepsilon$ intersects $\{y=y_f\}$ in $(x_{y_f}(\varepsilon),y_f)$ with
\begin{align*}
 x_{y_f}(\varepsilon) \equiv (1+ay_f)e^{-y_f^{-1}b}+\mathcal O(\varepsilon),
\end{align*}
for $\alpha -\alpha_H(0)=\mathcal O(\varepsilon)$. 
We then (re-)define $\Lambda$ as
\begin{align*}
  \Lambda =\{(x,y)\vert x\in [x_{y_f}(\varepsilon)-\nu,x_{y_f}(\varepsilon)+\nu],\, y=y_f\},
\end{align*}
with $\nu>0$,
and consider the following mapping 
\begin{align*}
 P_\varepsilon:\quad \Gamma\rightarrow \Lambda,
\end{align*}
obtained by the forward flow of \eqref{eq.aircraftInfty}. We will then prove the following:
\begin{theorem}\label{thm.aircraft}
Consider $0<\varepsilon\ll 1$. Then for $\nu>0$ sufficiently small the mapping $P_\varepsilon$ is a strong contraction, satisfying the following estimates
\begin{align*}
 P_\varepsilon(x) = x_{y_f}(\varepsilon)+\mathcal O(e^{-c/\varepsilon}),\quad P_\varepsilon'(x) = \mathcal O(e^{-c/\varepsilon}).
\end{align*}

%
\end{theorem}

This result provides the transition from the canard cycles with head to those without head by enabling a description of the backward orbit $\hat \gamma_\varepsilon$ for $u<\bar c$ (recall \eqref{eq.uBarC} and Fig. \ref{Caircraft}) from $\Gamma$ up until the section $\Lambda$. Our approach describes the geometry of the transition.
We will focus on the details of subsets of $\Gamma$ with $x=o(1)$ with respect to $\varepsilon$. Initial conditions with $x=\mathcal O(1)$ but negative will be briefly discussed in the final section \ref{sec.final}. We will proceed as described in \mthref{method} 
by introducing the following function
\begin{align}
 q = (1+ay)e^{-y^{-1}b},\label{eq.qAircraft}
\end{align}
as a new dynamic variable
and consider the extended system:
\begin{align}
 \dot x&=y\left(\varepsilon(1+\alpha y)+x\left(x-q\right)\right),\label{eq.aircraftQ}\\
 \dot y&=y^2 \left(x-q\right),\nonumber\\
 \dot q&=q\left(x-q\right)\left(b+\frac{ay^2}{1+ay}\right).\nonumber
\end{align}
This system is obtained by differentiating \eqref{eq.qAircraft} with respect to time and using \eqref{eq.aircraftInfty} and \eqref{eq.qAircraft} to eliminate the exponential $e^{-y^{-1}b}$, following steps 2 and 3 of \mthref{method}. The set 
\begin{align}
 Q=\{(x,y,q)\vert q=(1+ay)e^{-y^{-1}b}\},\label{eq.setQAircraft}
\end{align}
obtained from \eqref{eq.qAircraft} is by construction an invariant set of \eqref{eq.aircraftQ}. Again it is implicit in \eqref{eq.aircraftQ} and we can invoke it when needed. We could also have set $q=e^{-y^{-1}b}$, the analysis would be almost identical. However, some of the resulting expressions simplify using \eqref{eq.qAircraft}. In particular, $S$ just reads:
\begin{align}
 S:\quad x = q. \label{eq.SaQ}
\end{align}
The set $S$ in \eqref{eq.SaQ} is a set of critical points of \eqref{eq.aircraftQ}$_{\varepsilon=0}$. It is nonhyperbolic for $x=q=0$, but now the system is algebraic to leading order. To deal with this loss of hyperbolicity we may therefore consider the following blowup:
\begin{align}
 q = r\bar q,\,x=r\bar x,\,\varepsilon = r^2 \bar \epsilon,\quad r\ge 0,\,(\bar q,\bar x,\bar \epsilon)\in S^2.\label{eq.aircraftBU}
\end{align}
Notice that $y$ (as promised) is not part of the transformation \eqref{eq.aircraftBU} so geometrically $x=q=\varepsilon=0$ is blown-up to a cylinder of spheres: $(y,(\bar q,\bar x,\bar \epsilon))\in \overline{\mathbb R}_+\times S^2$. The transformation \eqref{eq.aircraftBU} gives a vector-field $\overline X$ on the blowup space by pull-back. Here $\overline X_{r=0}=0$. However, it is the desingularized vector-field $\widehat X=r^{-1}\overline X$ with $\widehat X\vert_{r=0}\ne 0$ that we study in the following. \ed{For simplicity, we just focus on the scaling chart:
\begin{align}
 \bar \varepsilon&=1:\quad q=r_2q_2,\,x=r_2x_2,\,\varepsilon=r_2^2.\nonumber
\end{align}
The details of the entry chart $\bar q=1$ is similar to the analysis in section \ref{sec.modelFlat} and therefore left out. The scaling chart $\bar \epsilon=1$ then focuses on $x=\mathcal O(r_2=\sqrt{\varepsilon})$. As we shall see, this scaling captures (upon further blowup) the transition from canards without heads to those with head. 
}

\subsection{The scaling chart}\label{sec.airKappa2}
Insertion gives the following system
\begin{align}
 \dot x_2&=y\left((1+\alpha y)+x_2\left(x_2-q_2\right)\right),\label{eq.aircraftQ2}\\
 \dot y&=y^2 \left(x_2-q_2\right),\nonumber\\
 \dot q_2&=q_2\left(x_2-q_2\right)\left(b+\frac{ay^2}{1+ay}\right),\nonumber
\end{align}
after division by $$r_2=\sqrt{\varepsilon}.$$ 
Let
\begin{align}
 U_2=\left\{(x_2,y,q_2)\vert x_2\in [-\xi^{-1},\xi^{-1}],\,y\in [0,\chi],\,q_2 = [0,\mu^{-1}]\right\},\label{eq.setU2Air}
\end{align}
be a closed box with side lengths $2\xi^{-1}$, $\chi$, $\mu^{-1}$ fixed with respect to $0<\varepsilon\ll 1$. Note that \eqref{eq.aircraftQ2} is independent of $r_2$ (and hence $\varepsilon$). The small parameter therefore only enters through the invariance of the set $Q$ \eqref{eq.setQAircraft}. Indeed, in the scaling chart, $\bar \varepsilon=1$, the set $Q$ becomes:
\begin{align*}
Q_2 = \{(x_2,y,q_2)\vert \sqrt{\varepsilon} q_2 = (1+ay)e^{-y^{-1}b}\}.
\end{align*}
We shall in the following consider initial conditions within $\Gamma_2\cap Q_2$ where 
\begin{align}
 \Gamma_2 = \{(x_2,y,q_2)\in U_2\vert q_2 = \mu^{-1}\},\label{eq.Gamma2Negative}
\end{align}
is the face of the box $U_2$ \eqref{eq.setU2Air} with $q_2=\mu^{-1}$. 
Note that
\begin{align}
 \Gamma_2\cap Q_2=\{(x_2,y,q_2)\in U_2\vert q_2 = \mu^{-1},\,y = \mathcal O(\ln^{-1} \varepsilon^{-1})\},\label{eq.initKappa2}
\end{align}
as $\varepsilon\rightarrow 0$. 
For $\varepsilon=0$ we therefore have $y=0$ on $\Gamma_2 \cap Q_2$. Setting $y=0$ in \eqref{eq.aircraftQ2} gives:
\begin{align}
 \dot x_2 &= 0,\label{eq.eqx2q2}\\
 \dot y &=0,\nonumber\\
 \dot q_2 &=q_2(x_2-q_2).\nonumber
\end{align}
Here 
\begin{align}
I_2=\{(x_2,y,q_2)\in U_2\vert y=0,\,q_2=0\},\label{eq.I2}
\end{align}
undergoes a transcritical bifurcation at $x_2=0$, as a set of critical points of \eqref{eq.eqx2q2}, going from asymptotically stable for $x_2<0$ to unstable for $x_2>0$. This produces
\begin{align}
C_2 = \{(x_2,y,q_2)\in U_2\vert y=0,\,q_2=x_2,\,x_2>0\}, \label{eq.C2}
\end{align}
at $x_2=0$, as a new set of critical points.

\ed{We now fix $\nu>0$ small and divide the proof of Theorem \ref{thm.aircraft} into three parts, considering initial conditions within $\Gamma_2\cap Q_2$ satisfying: 
\begin{itemize}
 \item[(a)] $x_2\in [\nu,\xi^{-1}]$;
 \item[(b)] $x_2\in [-\xi^{-1}, -\nu]$;
 \item[(c)] and finally $x_2\in (-\nu,\nu)$. 
\end{itemize}
We consider each of the cases in the following sections. }

\subsection*{Case (a): $x_2\ge \nu$}
We have:
\begin{lemma}
The critical set $C_2$ of \eqref{eq.eqx2q2} is normally hyperbolic. Let $\mu<\xi$ in $U_2$ \eqref{eq.setU2Air}. Then any compact subset $C_2\cap \{x_2\ge \nu\}$, $\nu>0$ small, therefore perturb into an attracting, locally invariant, center manifold manifold
\begin{align}
 S_2 =\{U_2 \vert q_2 = x_2 + ym_2(x_2,y),\,x_2 \in [\nu,\xi^{-1}],\,y \in [0,\chi]\},\label{eq.S2lemma}
\end{align}
for $\chi$ sufficiently small, with $m_2$ smooth and satisfying:
\begin{align*}
 m(x_2,0) =- \frac{1}{bx_2}.
\end{align*}

\end{lemma}
\begin{proof}
 Straightforward calculation. The linearization about a point in $C_2$ gives $-bx_2\le -b\nu<0$ as a single non-zero eigenvalue. This provides the existence of a center manifold $S_2$, local in $y$. 
\end{proof}
The center manifold $S_2$ is (through a trivial extension into chart $\kappa_1$) an enlargement of the slow manifold of \eqref{eq.SaQ}, upon scaling back down and restricting by $Q$, up until a neighborhood of $(x,y)=0$ that scales like
\begin{align*}
 (x,y) = \left(\mathcal O(\sqrt{\varepsilon}),\mathcal O(\ln^{-1} \varepsilon^{-1}))\right).
\end{align*}
The reduced problem on $S_2\cap Q_2$ gives 
\begin{align*}
 \dot x_2  &= 1+\mathcal O(y),\quad y=\mathcal O(\ln^{-1}\varepsilon^{-1}),
\end{align*}
after division by $y$ on the right hand side. Hence $x_2$ is increasing on $S_2$ for $\varepsilon>0$ sufficiently small. This enable us to guide initial conditions within $\Gamma_2 \cap Q_2$ with $x_2\ge \nu$ along $S_2$ and eventually back (by a simple passage through the chart $\bar q=1$) to the section $\Sigma$. The strong uniform contraction along the slow manifold establishes the result of Theorem \ref{thm.aircraft} for these set of initial conditions. Case (a) therefore describes an extension to large (since $y=\mathcal O(\ln^{-1}\varepsilon^{-1})$) canard cycles without head.

\subsection*{Case (b): $x_2\le -\nu$}
On the other hand, for $x_2\le -\nu$, say, we obtain a delayed stability phenomenon due to the attraction (repulsion) of the invariant set $I_2$ \eqref{eq.I2} for $x_2<0$ ($x_2>0$). 
We describe the delayed stability in the following lemma:
\begin{lemma}\label{finalProp}
 Consider $0<\varepsilon\ll 1$ and initial conditions in $\Gamma_2\cap Q_2$ with $x_2\le -\nu $, $\nu>0$. Then for $\mu$ sufficiently large, the forward flow of \eqref{eq.aircraftQ2} gives rise to a first return mapping $x_2\mapsto x_2^+(x_2)$ on $\Gamma_2\cap Q_2$ which satisfies:
 \begin{align*}
  x_2^+(x_2) = -x_2 + o(1),\quad (x_2^+)'(x_2) = -1 + o(1),
 \end{align*}
as $\varepsilon\rightarrow 0$. 
\end{lemma}
\begin{proof}
 A proof of this statement follows the original proof of the delayed stability phenomenon in planar slow-fast systems studied in \cite{schecter}. This approach to the problem basically uses appropriate lower and upper solutions to properly bound the motion of the fast variable. To translate this into the current context we need to bound $q_2$. For this we first bound $y$. We note the following: For $q_2=0$ we can write \eqref{eq.aircraftQ2} as
 \begin{align}
  y'(x_2) = \frac{yx_2}{1+x_2^2+\alpha y}.\label{eq.dy2dx2}
 \end{align}
 The solution through $y(0)=y_0$ is
 \begin{align}
y = y_0\frac{\sqrt{(\alpha y_0+1)^2+x_2^2(2\alpha+1)}+\alpha y_0}{1+2\alpha y_0}=y_0 \sqrt{x_2^2+1}(1+\mathcal O(y_0)).\label{eq.y2Sol}
 \end{align}
Therefore for $\delta>0$ sufficiently small we have that $y$ of \eqref{eq.aircraftQ}, with initial conditions from \eqref{eq.Gamma2Negative}, is bounded as
\begin{align*}
 y_0\sqrt{x_2^2+1}(1-\delta) \le y\le y_0 \sqrt{x_2^2+1}(1+\delta).
\end{align*}
Here $y_0=\mathcal O(\ln^{-1}\varepsilon^{-1})$ cf. the initial conditions \eqref{eq.Gamma2Negative} and $y_0$ will play the role of the small parameter. From here we can bound the fast variable by considering the following two upper $\bar q_2$ and lower $\underline q_2$ solutions  obtained from the following equations:
\begin{align*}
 \bar q_2'(x_2)&=\frac{\bar q_2(x_2+\delta)}{y_0 \sqrt{x_2^2+1}(1-\delta)},\\
 \underline q_2'(x_2)&=\frac{\underline q_2(x_2-\delta)}{y_0 \sqrt{x_2^2+1}(1+\delta)},
\end{align*}
after possibly decreasing $\delta$ further. 
Proceeding as in \cite{schecter} gives the desired result.
\end{proof}

\begin{figure}
\begin{center}
\includegraphics[width=.7\textwidth]{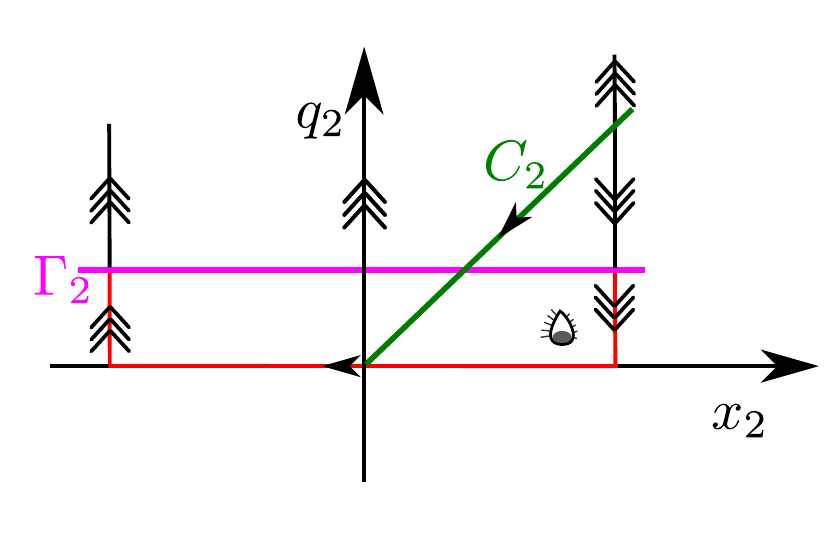}
 \caption{Illustration of the delayed stability that gives rise to canards with head. The return mapping to $\Gamma_2$ is described by Lemma \ref{finalProp}. We have here used the original orientation of time also used in Fig. \ref{aircraftNew}. Hence, the arrows should be reversed when time is as in \eqref{eq.aircraftQ2}.}
\label{aircraftKappa2}
 \end{center}
              \end{figure}

From $(x_2,y) = (x_2^+(x_2)+o(1),\mathcal O(\ln^{-1} \varepsilon^{-1}))$ we use the attraction of $S_2$ to follow the trajectory up to $\Sigma$. This proves the statement of Theorem \ref{thm.aircraft} for this set of initial conditions. Case (b) therefore describes large canards with small heads through delayed stability. See also Fig. \ref{aircraftKappa2} where we, for easing the comparison with Fig. \ref{aircraftNew}, use the original direction of time.

\subsection*{Case (c): $x_2\in (-\nu,\nu)$}
\ed{In the following we describe the forward flow of the initial conditions within $\Gamma_2\cap Q_2$ with $x_2\in (-\nu,\nu)$. For this we first drop the subscripts in \eqref{eq.aircraftQ2} so that
\begin{align}
 \dot x&=y\left((1+\alpha y)+x\left(x-q\right)\right),\label{eq.xyq}\\
 \dot y&=y^2 \left(x-q\right),\nonumber\\
 \dot q&=q\left(x-q\right)\left(b+\frac{ay^2}{1+ay}\right),\nonumber
\end{align}
 and consider the subsequent blowup:
\begin{align}
 x = \rho \bar x,\quad y = \rho^2 \bar y,\quad q = \rho \bar q,\quad (\rho,(\bar x,\bar y,\bar q)) \in \overline{\mathbb R}_+\times S^2,\label{eq.FinalBlowup}
\end{align}
of the nonhyperbolic point $(x,y,q)=0$ of \eqref{eq.xyq} to a sphere $(\rho,(\bar x,\bar y,\bar q))\in \{0\}\times S^2$. Here desingularization is obtained through division by $\rho$. 
We consider the following charts:
\begin{align*}
 \kappa_1:\quad \bar q&=1:\quad x=\rho_1x_1,\,y=\rho_1^2 y_1,\, q=\rho_1,\nonumber\\
 \kappa_2:\quad \bar x&=-1:\quad x=-\rho_2,\,y=\rho_2^2 y_2,\,q=\rho_2q_2,\nonumber\\
 \kappa_3:\quad \bar y&=1:\quad x=\rho_3x_3,\,y=\rho_3^2,\,q_3=\rho_3q_3,\nonumber\\
 \kappa_4:\quad \bar x&=1:\quad \quad x=\rho_4,\,y=\rho_4^2 y_4,\,q=\rho_4q_4,\nonumber,
\end{align*}
successively. 
The chart $\kappa_1$ is most important. Here we obtain the equations
\begin{align}
 \dot \rho_1 &= \rho_1 (x_1-1)\left[b+\frac{a\rho_1^4 y_1^2}{1+a\rho_1^2 y_1}\right],\label{eq.airKappa1}\\
 \dot x_1 &=y_1 \left(1+\alpha \rho_1^2 y_1+\rho_1^2 x_1(x_1-1)\right)-x_1(x_1-1)\left[b+\frac{a\rho_1^4 y_1^2}{1+a\rho_1^2 y_1}\right],\nonumber\\
 \dot y_1 &=-y_1(x_1-1)\left(2\left[b+\frac{a\rho_1^4 y_1^2}{1+a\rho_1^2 y_1}\right]-y_1\rho_1^2\right).\nonumber
\end{align}
Note that initial conditions within $\Gamma_2\cap Q_2$ \eqref{eq.initKappa2} belongs to the invariant set $\{y_1=0\}$ for $\varepsilon = 0$ with $\rho_1=\mu^{-1}$ in this chart. 
We then prove the succeeding lemma in Appendix \ref{appD}:
\begin{lemma}\label{lem.newKappa1}
 The following holds for \eqref{eq.airKappa1}:
 \begin{itemize}
  \item[$1^\circ$] The equilibrium 
  \begin{align}
  (\rho_1,x_1,y_1)=0,\label{eq.originK1}
  \end{align}
  is hyperbolic, the linearization having eigenvalues: $-b,\,b$ and $2b$. Let $U_1$ be a small neighborhood of the origin $(\rho_1,x_1,y_1)=0$. Then
  \begin{align*}
   W^s(0)\cap U_1&=\{(\rho_1,x_1,y_1)\in U_1\vert x_1=y_1=0,\,\rho_1\ge 0\},\\   
   W^u(0)\cap U_1&=\{(\rho_1,x_1,y_1)\in U_1\vert \rho_1=0\}.
  \end{align*}
  Within $W^u(0)\subset \{\rho_1=0\}$ there exists a unique smooth strong unstable manifold:
  \begin{align*}
  W^{uu}(0)\cap U_1= \{(\rho_1,x_1,y_1)\in U_1\vert \rho_1 = 0, y_1 = b x_1 + \mathcal O(x_1^2), 
  \end{align*}
tangent to the strong eigenvector $(0,1,b)$ of the eigenvalue $2b$ at the equilibrium \eqref{eq.originK1}.
\item[$2^\circ$] The strong unstable manifold is contained within the first quadrant of the $(x_1,y_1)$-plane and is forward asymptotic to the equilibrium
\begin{align}
  L_1:\,\rho_1=0,\,x_1=1,\,y_1=0,\label{eq.L1K1}
  \end{align}
approaching $L_1$ along the vector $(0,1,b)$. 
\item[$3^\circ$] The equilibrium \eqref{eq.L1K1} is partially hyperbolic, the linearization having eigenvalues $0,\,0$ and $-b$. Let $V_1$ be a small neighborhood of \eqref{eq.L1K1}. Then within $V_1$ there exists a smooth attracting center manifold 
\begin{align}
 S_1:\,x_1 = 1+y_1m_1(\rho_1,y_1),\quad m_1(\rho_1,y_1) = b^{-1}+\mathcal O(y_1),\label{eq.S1Expr}
\end{align}
containing 
$$C_1:\,\rho_1>0,\, x_1=1,y_1=0,$$ as a line of fix-points, and
\begin{align*}
 D_1 = W^{uu}(0)\cap V_1\subset \{\rho_1=0\},
\end{align*}
as a geometrically unique center sub-manifold. 
 \end{itemize}
\end{lemma}

We illustrate our findings in Fig. \ref{airExtra}. Here we have used the original orientation of time also used in Fig. \ref{aircraftNew}. The details of the charts $\kappa_{2-4}$ are delayed to the Appendix \ref{appD}. Notice that
\begin{itemize}
\item[(i)] $S_1$ in chart $\kappa_1$ is the continuation of $S_2$ from above (see \eqref{eq.S2lemma}) onto the blowup sphere as an attracting center manifold;
\item[(ii)] $W^{uu} = W^{uu}(0)$ (cf. $3^\circ$ of Lemma \ref{lem.newKappa1}) locally at \eqref{eq.L1K1} coincides (by choice) with the intersection, $S_1\cap \{\rho_1=0\}$, of $S_1$ (or more accurately $\bar S$) with the blowup sphere $\{\rho=0\}$ and acts as a separatrix of canards for $\varepsilon=0$: On one side they are without head whereas on the other side they have head (indicated by the eye). 
\item[(iii)] The orbits in blue, purple, red and green in Fig. \ref{airExtra} correspond to different initial conditions on the set $\Gamma_2\cap Q_2$ \eqref{eq.initKappa2} for $0<\varepsilon\ll 1$. The blue orbit takes a large excursion around the blowup sphere (visualized in $\kappa_{2-4}$ of Fig. \ref{airExtra}), following the equator: $\bar q=0$ of the blowup sphere for an extended period of time. This orbit represents a canard with head (like a perturbed version of the red orbit in Fig. \ref{aircraftKappa2}). The purple and the red orbit have smaller heads in comparison. Finally, the green orbit represents a canard cycle without head.
\end{itemize}}
\begin{figure}
\begin{center}
\includegraphics[width=.99\textwidth]{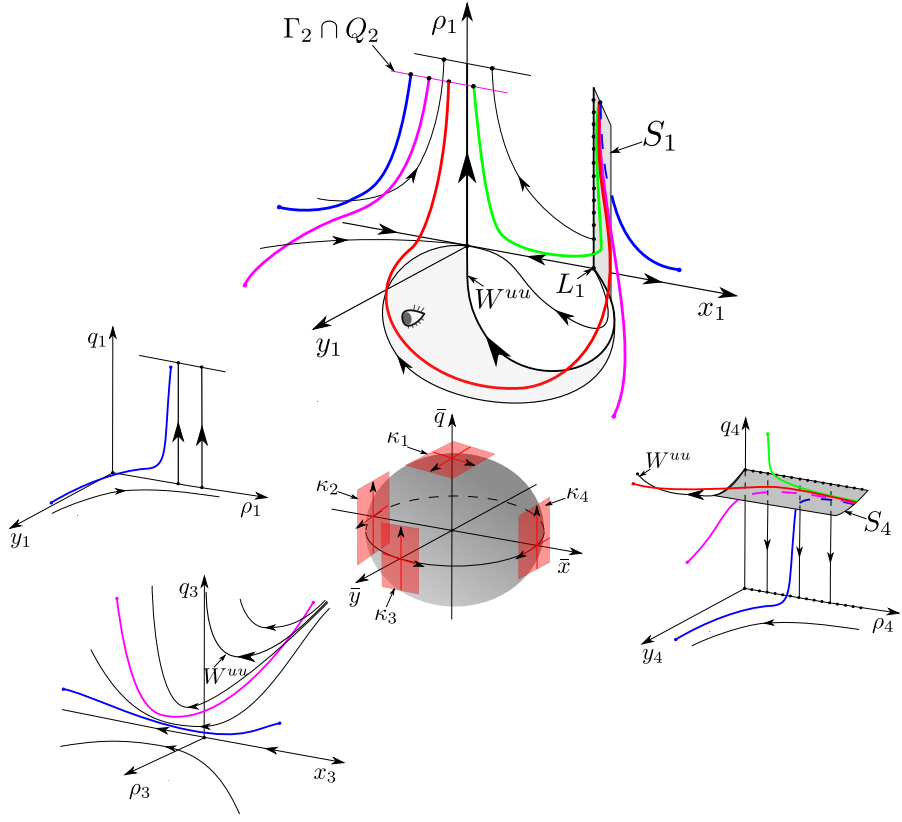}
 \caption{Illustration of the transition from canards without head to canards with head described by the blowup \eqref{eq.FinalBlowup} of $(x,y,q)=0$ in \eqref{eq.xyq}. We have here used the original orientation of time also used in Fig. \ref{aircraftNew}. Hence, the arrows should be reversed when time is as in \eqref{eq.xyq}. In the limit $\varepsilon=0$, the strong unstable manifold (strong stable manifold in the time used here) $W^{uu}=W^{uu}(0)$ acts as a separatrix of canards: On one side they are without head whereas on the other side they have head (indicated by the eye).   }
\label{airExtra}
 \end{center}
              \end{figure}
\subsection{Final remarks}\label{sec.final}
Consider now the extended system \eqref{eq.aircraftQ} on the original phase space $\{(x,y,q)\}$. Then the initial conditions on $\Gamma$ with $x\le -\nu$, $\nu>0$ small but fixed, can be followed into chart $\bar x=-1$ of \eqref{eq.aircraftBU} and then subsequently into charts $\kappa_2$ and $\bar x=1$. These orbits will for decreasing values of $x$ remain closer to $\bar q=0$ for a longer period of time before jumping towards the slow manifold $S$. Eventually the fast jump along the critical fibers will occur so that the forward orbit no longer intersects $\Lambda$ close to $S_\varepsilon\cap \Lambda$. In fact, ultimately the forward orbit does not intersect $S_\varepsilon$ at all. We skip the details of this because this can be obtained without the introduction of $q$. For the original model, which is just obtained by reversing time, this gives rise to the transition from canards with head to relaxation oscillations that (a) follow the attracting branch $\hat S_\varepsilon$, then (b) jump near the fold $F$ towards infinity. Here (c) the orbit follows $y=0$ until it (d) takes off along a critical fiber in the chart $\bar x=-1$ to return to $\hat S_\varepsilon$. 

\section{Conclusion}
We have presented a novel approach to deal with flat slow manifolds that appear in slow-fast systems. The basic idea of this method is to embed the system into a higher dimensional version for which the standard blowup approach, in the formulation of Krupa and Szmolyan, can be applied to deal with the loss of hyperbolicity. In this paper, we did not aim to put our approach into a general framework (this should be part of future work) but instead demonstrated its use on two examples: Regularization of PWS systems using $\tanh$ and a model of aircraft ground dynamics. In the future, it would also be interesting to pursue applications of our approach to areas outside the realm of the classical geometric singular perturbation theory.
\section*{Acknowledgement} {The author would like to thank C. Kuehn for pointing me in the direction \cite{aircraft}, Elena Bossolini, Morten Br{\o}ns and Peter Szmolyan for useful discussions, and finally Stephen Schecter for providing valuable feedback. }

\appendix
\section{Proof of Theorem \ref{thm.Bonet} using blowup}\label{appA}
We consider the following system
\begin{align}
 \dot x &=\varepsilon (1+\phi(\hat y)),\label{eq.eqApp}\\
 \dot{\hat y}&=2x(1+\phi(\hat y))+1-\phi(\hat y),\nonumber
\end{align}
with $\phi\in C_{ST}^{n-1}$, $n\ge 2$ (see Definition \ref{STphi}).
Here 
\begin{align*}
 S_a:\,\phi(\hat y) = \frac{1+2x}{1-2x},\quad x<0,
\end{align*}
is an attracting critical manifold. It loses hyperbolicity at $(x,\hat y)=(0,1)$ in a  fold (degenerate for $n\ge 3$, in particular, a cusp for $n=3$). Indeed, the linearization of \eqref{eq.eqApp} about $S_a$ gives an eigenvalue
\begin{align*}
 (2x-1)\phi'(\hat y).
\end{align*}
This eigenvalue vanishes at $\hat y=1$ since $\phi'(\hat y)=0$ by $3^\circ$ in Definition \ref{STphi}. For $x\in [-c,-\rho]$, $\rho$ sufficiently small, the critical manifold $S_a$ perturbs to a Fenichel slow manifold $S_{a,\epsilon}$. We will initially seek to guide $S_{a,\varepsilon}$ up until $\hat y=1$. For this we will use the following expansion of $\phi(\hat y)$ about $\hat y=1$
\begin{align*}
 \phi(\hat y) = 1-\phi^{[n]}(-\tilde y)^{n}(1+\tilde yR(\tilde y))),
\end{align*}
where
\begin{align*}
 \hat y=1+\tilde y,\quad \phi^{[n]}=\frac{(-1)^{n+1}}{n!}\phi^{(n)}(1)>0.
\end{align*}
Here $R(\tilde y)$ is smooth by Taylor's theorem. This expression is valid for $\tilde y\in (-2, 0]$ and follows directly from Definition \ref{STphi}. This gives the following equations:
 \begin{align}
 \dot x &=\varepsilon ,\label{eq.eqApp1}\\
 \dot{\tilde y}&=2x+\frac12 \phi^{[n]}(-\tilde y)^{n}(1+\tilde y Q(\tilde y))).\nonumber
\end{align}
for some smooth $Q(\tilde y)$, 
after division of the right hand side by 
\begin{align*}
1+\phi(\hat y)=2-\phi^{[n]}(-\tilde y)^n (1+\tilde yR(\tilde y)))>0\quad \text{within}\quad \tilde y\in (-2,0).
\end{align*}
Now we apply the following blowup:
\begin{align*}
x=r^{n}\bar x,\, \tilde y = r \bar y,\,\varepsilon = r^{2n-1}\bar \varepsilon,\quad (r,(\bar x,\bar y,\bar \epsilon))\in \overline{\mathbb R}_+\times S^2,
\end{align*}
and consider the following two charts:
\begin{align*}
\kappa_1:&\quad \bar y=-1:\quad x=r_1^{n}x_1,\,\tilde y = -r_1,\,\varepsilon = r_1^{2n-1} \varepsilon_1,\\
\kappa_2:&\quad \bar \varepsilon=1:\quad x=r_2^{n}x_2,\,\tilde y = r_2 y_2,\,\varepsilon = r_2^{2n-1}.
\end{align*}
\subsection{Chart $\kappa_1$}
In this chart we obtain the following system:
\begin{align*}
 \dot x_1 &=\epsilon_1+x_1 \left(2x_1+\frac12 \phi^{[n]} (1-r_1Q(-r_1))\right),\\
 \dot r_1 &=-r_1 \left(2x_1+\frac12 \phi^{[n]}(1-r_1Q(-r_1))\right),\\
 \dot \epsilon_1&= (2n-1) \left(2x_1+\frac12 \phi^{[n]} (1-r_1Q(-r_1))\right) \epsilon_1,
\end{align*}
after desingularization through the division by $r_1^{n-1}$ on the right hand side. In this chart $S_a$ becomes
\begin{align*}
 S_{a,1}:\quad x_1 = -\frac14 \phi^{[n]} (1-r_1Q(-r_1)),\quad r_1\ge 0.
\end{align*}
We consider the following box
\begin{align*}
 U_1=\{(x_1,r_1,\epsilon_1)\vert x_1\in [-\xi^{-1},0],\,r_1 \in [0,\rho],\,\epsilon_1=[0,\nu]\},
\end{align*}
with side lengths $\xi^{-1}$, $\rho$, and $\nu$. We take $\xi$, $\rho$ and $\nu$ sufficiently small. Center manifold theory applied to the equilibrium $x_1=-\frac14 \phi^{[n]}$, $r_1=\epsilon_1=0$ as a partially hyperbolic equilibrium gives the following:
\begin{proposition}
There exists a center manifold within $U_1$:
\begin{align*}
 M_1:\quad x_1 = -\frac14 \phi^{[n]} (1-r_1Q(-r_1))+\epsilon_1m_1(r_1,\epsilon_1),
\end{align*}
with
\begin{align*}
 m_1(r_1,\epsilon_1) = \frac{2}{\phi^{[n]}}+\mathcal O(r_1+\epsilon_1).
\end{align*}
The manifold $M_1$ is foliated by invariant hyperbolas $$E_1:\quad\varepsilon = r_1^{2n-1} \varepsilon_1,$$ with $\varepsilon=\text{const}.\ge 0$, and contains the critical manifold $S_{a,1}$ within $\epsilon_1=0$, as a set of critical points, and 
\begin{align*}
 C_{a,1}:\quad x_1 = -\frac14 \phi^{[n]}+\epsilon_1 m_1(0,\epsilon_1),
\end{align*}
within $r_1=0$, 
as a unique center sub-manifold.
\end{proposition}
\begin{proof}
 The proof of this is straightforward. The uniqueness of $C_{a,1}$ follows from the fact that $\dot \epsilon_1>0$ within $C_{a,1}\cap \{\epsilon_1>0\}$.
\end{proof}
\begin{remark}
As usual $M_1\cap E_1$ is $\mathcal O(e^{-c/\varepsilon})$-close to Fenichel's slow manifold $S_{a,\varepsilon,1}$ at $r_1=\rho$ and we shall therefore view $M_1\cap E_1$ as the extension of Fenichel's slow manifold. At $\epsilon_1=\nu$ we have $r_1=(\nu^{-1}\varepsilon)^{1/(2n-1)}$ within $E_1$ and therefore the proposition provides an extension of $S_{a,\varepsilon,1}$ satisfying:
\begin{align}
 S_{a,\varepsilon,1}\cap \{\epsilon_1=\nu\}:\quad x_1 = -\frac14 \phi^{[n]} +\nu m_1(0,\nu)+\varepsilon^{1/(2n-1)} Q_1(\varepsilon^{1/(2n-1)}),\label{eq.appSa1Eps}
\end{align}
with $Q_1$ smooth. In other words, $S_{a,\varepsilon,1}$ is $\varepsilon^{1/(2n-1)}$-smoothly close to $C_{a,1}$ at $\epsilon_1=\nu$. 
 \end{remark}
 We continue $S_{a,\varepsilon,1}$ into chart $\kappa_2$ in the following. For this we will use the closeness of $S_{a,\varepsilon,1}$ to $C_{a,1}$ and therefore guide $S_{a,\varepsilon,1}$ by following $C_{a,1}$. The change of coordinates between $\kappa_1$ and $\kappa_2$ is given as
 \begin{align*}
  x_2 = \epsilon_1^{-n/(2n-1)} x_1,\quad y_2 = -\epsilon_1^{-1/(2n-1)},\quad r_2=r_1\epsilon_1^{1/(2n-1)}.
 \end{align*}
We will denote $C_{a,1}$ and $S_{a,\varepsilon,1}$ by $C_{a,2}$ and $S_{a,\varepsilon,2}$, respectively, in chart $\kappa_2$.



\subsection{Chart $\kappa_2$}
Insertion gives
\begin{align}
 \dot x_2 &=1,\label{eq.appChart2}\\
 \dot y_2&=2x_2+\frac12 \phi^{[n]} (-y_2)^n(1+r_2y_2Q(r_2y_2)),\nonumber\\
 \dot r_2&=0,\nonumber
\end{align}
after desingularization through division of the right hand side by $r_2^{n-1}$. Here $r_2=\varepsilon^{1/(2n-1)}$. We have
 \begin{lemma}
  The forward flow of $C_{a,2}\subset \{r_2=0\}$ intersects $\{y_2=0\}$ in
  \begin{align*}
   (x_2,y_2,r_2) = (c_x \eta(n),0,0),
  \end{align*}
where $\eta(n)$ only depends upon $n$ and 
\begin{align}
 c_x = \left(\frac{2}{\phi^{[n]}}\right)^{1/(2n-1)}.\label{eq.appCx}
\end{align}
 \end{lemma}
\begin{proof}
We scale $x_2$ and $y_2$ by introducing:
\begin{align*}
 x_2 = c_x u,\quad y_2 = c_y v,
\end{align*}
with $c_x$ as in \eqref{eq.appCx} and
\begin{align*}
 c_y = - 2c_x^2.
\end{align*}
This transforms \eqref{eq.appChart2} into
\begin{align*}
 \dot u &=1,\\
 \dot v &=-u-v^n,
\end{align*}
for $r_2=0$ and $v\ge 0$, after scaling time by $c_x$. The result then follows from \cite[Proposition 3.10]{reves_regularization_2014}. 
\end{proof}
Now, using the $\varepsilon^{1/(2n-1)}$-smooth closeness of $S_{a,\varepsilon,2}$ to $C_{a,2}$ at $y_2=-\nu^{-1/(2n-1)}$, we can apply regular perturbation theory and blow back down to conclude the following:
\begin{proposition}
 The forward flow of Fenichel's slow manifold $S_{a,\varepsilon}$ intersects $\hat y=1$ in $(x,\hat y)=(x_\varepsilon,1)$ with
 \begin{align}
  x_\varepsilon = r_2^n \left(\frac{2}{\phi^{[n]}}\right)^{1/(2n-1)} \eta(n)(1+r_2 Q_2(r_2)),\label{eq.SaEpsHatYEq1}
 \end{align}
with $r_2=\varepsilon^{1/(2n-1)}$ and $Q(r_2)$ smooth. 
\end{proposition}
\subsection{Scaling down}
To continue $S_{a,\varepsilon}$ beyond $\hat y=1$ and towards $x=\delta$ we scale back down using \eqref{eq.haty} and return to the $(x,y)$-variables:
\begin{align*}
 \dot x &=1,\\
 \dot y &=2x.
\end{align*}
We then use $(x,y)=(x_\varepsilon,\varepsilon)$ as an initial condition and obtain, through simple integration,
\begin{align*}
 y_{\theta}(\varepsilon) = \theta^2+\varepsilon-x_\varepsilon^2,
\end{align*}
at $x=\theta$. This then completes the proof of Theorem \ref{thm.Bonet}.
\section{Proof of Theorem \ref{thm.tanh} by direction integration}\label{direct}
The system \eqref{eq.xhaty} with $\phi(\hat y)=\tanh (\hat y)$ can be written as 
 \begin{align*}
 \frac{dy}{dx} = 2x + \frac{1-\tanh(y \epsilon^{-1})}{1+\tanh(y \epsilon^{-1})} = 2x + e^{-2y\epsilon^{-1}},
\end{align*}
upon elimination of time and returning to $y$ through \eqref{eq.haty}. Integrating this gives
\begin{align*}
 y(x) = x^2 +\frac12 \varepsilon \ln \left(\sqrt{\frac{\pi}{2\varepsilon}} \left(\textnormal{erf}\,\left(\sqrt{2\varepsilon^{-1}} x\right)+C\right)\right),
\end{align*}
where $C$ is an integration constant and
\begin{align}
 \textnormal{erf}\,(u) = \frac{2}{\sqrt{\pi}}\int_0^u e^{-s^2}ds,\label{eq.errFunction}
\end{align}
is the Gauss error function satisfying
\begin{align*}
  \hspace{2cm}{\textnormal{erf}\,(u) = \pm 1-\frac{e^{-u^2}}{\sqrt{\pi}u}\left(1+\mathcal O(u^{-2})\right)\quad \text{for}\quad u\rightarrow \pm \infty}
 \hspace{2cm}{\text{(B.1)}_\pm}
\end{align*}
By (B.1)$_-$ it follows that the solution with $C=1$:
\begin{align*}
 y(x) = x^2 +\frac12 \varepsilon \ln \left(\sqrt{\frac{\pi}{2\varepsilon}} \left(\textnormal{erf}\,\left(\sqrt{2\varepsilon^{-1}} x\right)+1\right)\right),
\end{align*}
does not contain any fast components. It therefore represents a geometrically unique slow manifold. Now using (B.1)$_+$ we obtain 
\begin{align*}
 y(\theta) = \theta^2 + \varepsilon \left(\frac14 \ln \left(2\pi \varepsilon^{-1}\right)+\mathcal O\left(e^{-2\varepsilon^{-1} \theta^2}\right)\right),
\end{align*}
in agreement Theorem \ref{thm.tanh}. 

\section{Analysis of the charts $\kappa_{1-3}$ for the proof of Theorem \ref{thm.tanh}}\label{appC}

\subsection{Chart $\kappa_1$}
Inserting \eqref{eq.kappa1tanh} into \eqref{eq.extQ} gives the following equations:
\begin{align}
\dot r_1&=-2r_1(2x_1+1),\label{eq.eqnChartK1Tanh}\\
\dot x_1 &=\epsilon_1+2x_1(2x_1+1),\nonumber\\
 \dot{\hat \varepsilon} &=-\hat \varepsilon^2 (2x_1+1),\nonumber\\
 \dot \epsilon_1 &=4\epsilon_1( 2x_1+1),\nonumber
\end{align}
after desingularization through division of the right hand side by $r_1$. 
The critical manifold $S_{a}$ \eqref{eq.Saq} becomes 
\begin{align*}
 S_{a,1}:\quad x_1= -\frac{1}{2},\,\hat \varepsilon>0,\, r_1>0,\,\epsilon_1=0.
\end{align*}
Furthermore, the set \eqref{eq.setQTanh} becomes
\begin{align}
 Q_1:\quad  r_1={e^{-2\hat \varepsilon^{-1}}}.\label{eq.Q1Tanh}
\end{align}
Since $\varepsilon=r_1^2\epsilon_1$ the system \eqref{eq.eqnChartK1Tanh} possesses another invariant:
\begin{align*}
 E_1:\quad \varepsilon=r_1^2\epsilon_1.
\end{align*}
\begin{lemma}\label{lem.Ma1}
Let 
\begin{align}
\rho(\xi)={e^{-2\xi^{-1}}}, \label{eq.rhoDelta}
\end{align}
so that $\hat \varepsilon=\xi$ corresponds to $r_1=\rho(\xi)$ on $Q_1$ cf. \eqref{eq.qq}. 
Consider then the following box
\begin{align*}
 U_1=\bigg\{(r_1,x_1,\hat \varepsilon,\epsilon_1)&\vert r_1\in [0,\rho(\xi)],\,\,x_1\in [-\chi^{-1},\chi^{-1}],\,\hat \varepsilon\in [0,\xi],\\
 &\epsilon_1\in [0,\nu]\bigg\},
\end{align*}
with side lengths $\rho(\xi)$, $2\chi^{-1}$, $\xi$, and $\nu$. Then for $\nu$ and $\chi$ sufficiently small the following holds true: Within $U_1$ there exists an attracting center manifold: 
\begin{align}
 M_{a,1}:\quad x_1 = -\frac{1}{2}+\frac12 \epsilon_1(1+\mathcal O(\epsilon_1)). \label{eq.Ma1}
\end{align}
The manifold $M_{a,1}$ contains $S_{a,1}$ within $\epsilon_1=0$ as a set of fix points. The center sub-manifold 
\begin{align}
 C_{a,1}:\quad x_1 = -\frac12  +\frac12 \epsilon_1(1+\mathcal O(\epsilon_1)),\,r_1=0,\,\hat \varepsilon=0.\label{eq.Ca1}
\end{align}
is unique as a center manifold contained within $r_1=0,\,\hat \varepsilon=0$.
\end{lemma}
\begin{proof}
 The linearization \eqref{eq.eqnChartK1Tanh} about $x_1=-1,\,r_1=\epsilon_1=0$ gives $-2$ as a single non-trivial eigenvalue. The associated eigenvector, spanning the stable space, is purely in the $x_1$-direction. On the other hand, the $3D$ center space is spanned by three eigenvectors, purely in the direction of $r_1$ and $\hat \varepsilon$, respectively, and the eigenvector $(0,1,0,2)$. The result then follows from straightforward computations. The manifold $C_{a,1}$ is unique since it is overflowing.
\end{proof}

\begin{remark}\label{rem.Ma1}
 The set $M_{a,1}\cap Q_1\cap E_1$ is $\mathcal O(e^{-c/\varepsilon})$-close to $S_{a,\varepsilon}$ at $r_1=\rho(\xi)$. Notice also that $\epsilon_1=\nu$ corresponds to $r_1=\sqrt{\varepsilon \nu^{-1}}$ within $E_1$. Further restriction to $Q_1$ gives $\hat \varepsilon=\mathcal O(\ln^{-1} \varepsilon^{-1})$. The manifold $M_{a,1}\cap Q_1\cap E_1$ is therefore the continuation $S_{a,\varepsilon,1}$ of Fenichel's slow manifold $S_{a,\varepsilon}$ up to $\hat \varepsilon=\mathcal O(\ln^{-1} \varepsilon^{-1})$. We illustrate the geometry in Fig. \ref{tanhkappa1}.
  \begin{figure}
\begin{center}
\includegraphics[width=.7\textwidth]{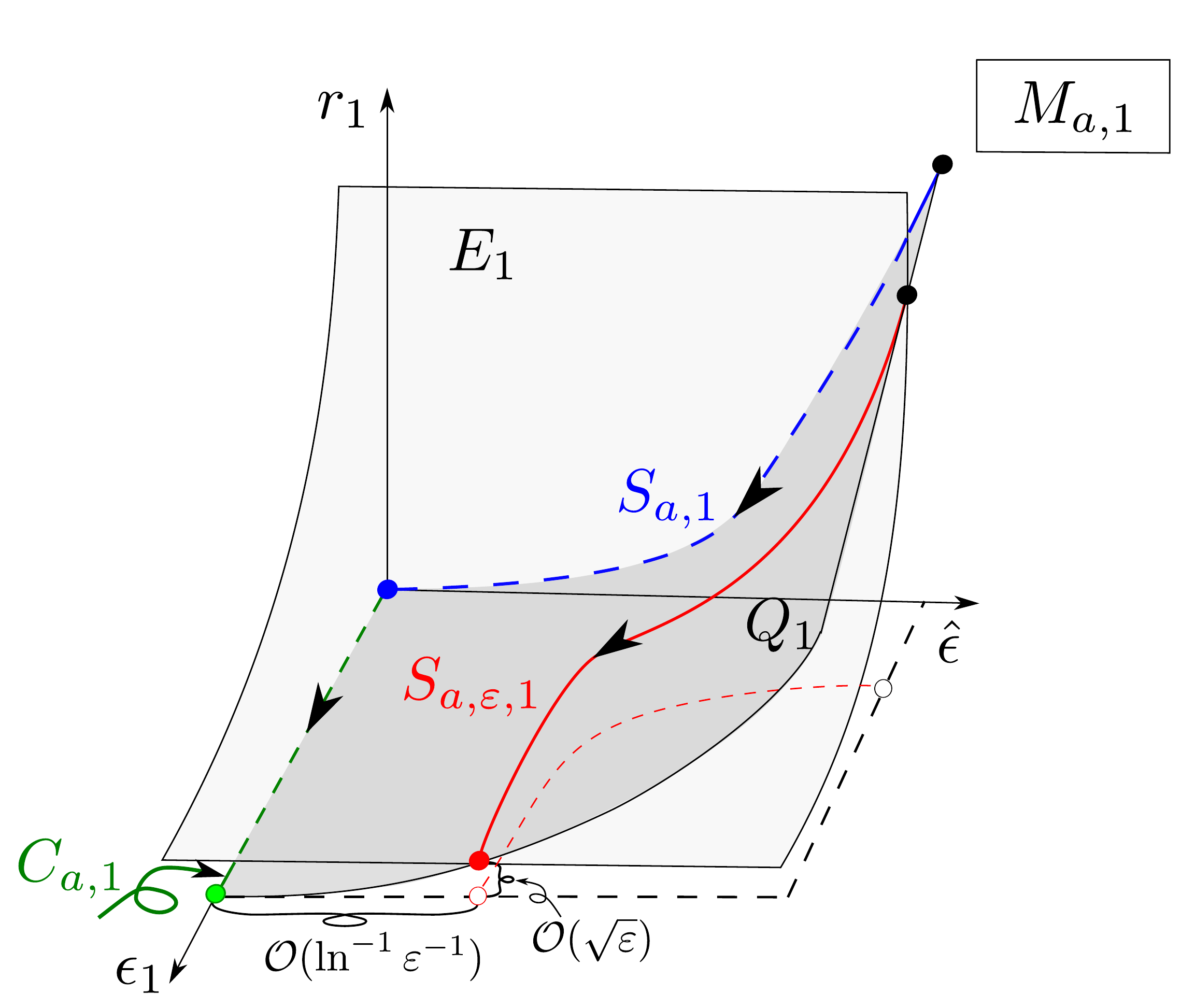}
 \caption{The reduced flow within $M_{a,1}$. The intersection of $M_{a,1}$ with $Q_1$ and $E_1$ provides an extension of the Fenichel slow manifold up to $\hat \varepsilon=\mathcal O(\ln^{-1}\varepsilon^{-1})$. }
\label{tanhkappa1}
 \end{center}
              \end{figure}
              
              In this particular case, note that the expression for $M_{a,1}$ can be taken to be independent of $r_1$. Since $C_{a,1}$ is unique this allows us to select a unique $M_{a,1}$ and therefore a unique $S_{a,\varepsilon,1}$. We shall apply this selection henceforth. 
              
              Finally, we note that the expression for $M_{a,1}$ is in agreement with the expression for $S_{a,\varepsilon}$ in \eqref{lemmaSaEpsTanh} at $\hat \varepsilon=\xi$, $r_1=\rho(\xi)=e^{-2\xi^{-1}}$. Indeed, ignoring $\mathcal O(\varepsilon^2)$-terms in \eqref{lemmaSaEpsTanh} gives
              \begin{align*}
               x = -\frac12 e^{-2\xi^{-1}} + \frac{\varepsilon}{2}e^{2\xi^{-1}},
              \end{align*}
              which in terms of $(x_1,\epsilon_1)$ can be written as 
              \begin{align*}
               r_1 x_1 = -\frac12 r_1 + \frac{r_1^2\epsilon_1}{2}r_1^{-1}= \frac12 r_1\left(-1 +\epsilon_1\right),
              \end{align*}
              using \eqref{eq.kappa1tanh} and 
             $r_1=e^{-2\xi^{-1}}$. 
              This then gives \eqref{eq.Ma1} (up to order $\mathcal O(\epsilon_1)$) upon division by $r_1$ on both sides.
\end{remark}


On $M_{a,1}$ we obtain the following reduced problem
\begin{align}
 \dot r_1 &=-2r_1,\label{eq.Mreduced}\\
\dot{\hat \varepsilon} &=-\hat \varepsilon^2,\nonumber\\
\dot \epsilon_1 &=4\epsilon_1,\nonumber
\end{align}
after division by $\epsilon_1(1+\mathcal O(\epsilon_1))$. This division desingularizes the dynamics within $S_{a,1}$, just as the passage to slow time desingularized the dynamics within the critical manifold. 

We describe \eqref{eq.Mreduced} in the following lemma:
\begin{lemma}
 Consider the reduced problem \eqref{eq.Mreduced} on $M_1$ and the mapping $$P_1:\quad (\hat \varepsilon,\epsilon_1)\mapsto (r_1^+,\hat \varepsilon^+),$$ from $\{(r_1,\hat \varepsilon,\epsilon_1)\vert r_1=\rho(\xi)\}$ to $\{(r_1,\hat \varepsilon,\epsilon_1)\vert \epsilon_1=\nu\}$ obtained by the forward flow. Then
 \begin{align}
  P_1(\xi,\rho(\xi)^{-2}\varepsilon) = 
\begin{pmatrix}
                                     \sqrt{\varepsilon \nu^{-1}}\\
                                                                        4 \ln^{-1} (\nu \varepsilon^{-1})
                                    \end{pmatrix}\label{eq.P1Out}
 \end{align}
In particular, the image $P_1(\xi,\rho(\xi)^{-2}\varepsilon)$ converges to the intersection
\begin{align}
C_{a,1}\cap\{\epsilon_1=\nu\}:\quad r_1=0,\,\hat \varepsilon=0,\,x_1 = -\frac12+\frac12 \nu(1+\mathcal O(\nu)),\,\epsilon_1=\nu,\label{eq.C1Out}
\end{align}
as $\varepsilon\rightarrow 0$. 
\end{lemma}
\begin{proof}
From the $r_1$- and the $\epsilon_1$-equation we obtain a travel time of $T=\frac14 \ln \left(\rho^2 \nu \varepsilon^{-1}\right)$. Integrating the $\hat \varepsilon$-equation then gives
\begin{align*}
 \hat \varepsilon(T) = \frac{1}{\xi^{-1}+T}=\left(\xi^{-1} +\frac14 \ln \left(\rho^2 \nu \varepsilon^{-1}\right)\right)^{-1}.
\end{align*}
Using \eqref{eq.rhoDelta} this then simplifies to
                                    \begin{align*}
                                     4 \ln^{-1} (\nu \varepsilon^{-1}).
                                    \end{align*}
\end{proof}

We continue the point \eqref{eq.P1Out} forward in time by moving to chart $\kappa_2$. The coordinate change between $\kappa_1$ and $\kappa_2$ is easily obtained from \eqref{eq.kappa1tanh} and \eqref{eq.kappa2tanh}. It is given as
\begin{align}
 q_2 = 1/\sqrt{\epsilon_1},\quad x_2 = x_1/\sqrt{\epsilon_1},\quad r_2=r_1 \sqrt{\epsilon_1},\label{eq.k1k2}
\end{align}
valid for $\epsilon_1>0$. 
\subsection{Chart $\kappa_2$}
Inserting \eqref{eq.kappa2tanh} into \eqref{eq.extQ} gives
\begin{align}
 \dot x_2 &= 1,\label{eq.eqChartK2Tanh}\\
 \dot{\hat \varepsilon} &=-\hat \varepsilon^2 (2x_2+q_2),\nonumber\\
 \dot q_2 &=-2q_2(2x_2+q_2),\nonumber\\
 \dot r_2 &=0,\nonumber
\end{align}
after division by $r_2$. In chart $\kappa_2$ the point \eqref{eq.P1Out}, using \eqref{eq.k1k2}, becomes:
\begin{align}
\hat \varepsilon=4 \ln^{-1} (\nu \varepsilon^{-1}),\,x_2 = \nu^{-1/2}\left(-\frac12 + \frac12 \nu(1+\mathcal O(\nu))\right),\,q_2=\nu^{-1/2},\label{eq.P2In}
\end{align}
and $r_2=\sqrt{\varepsilon}$. We will in this section guide \eqref{eq.P2In} up until the section $\{x_2=\eta^{-1/2}\}$ using the forward flow of \eqref{eq.eqChartK2Tanh}. For this we first note that \eqref{eq.C1Out} within $\kappa_2$ becomes
\begin{align*}
x_2 = \nu^{-1/2}\left(-\frac14+\nu(1+\mathcal O(\nu))\right),\,q_2=\nu^{-1/2},
\end{align*}
contained within $r_2=0,\hat \varepsilon=0$. We therefore consider $r_2=0,\,\hat \varepsilon=0$ and obtain the following system
\begin{align*}
 \frac{dq_2}{dx_2} &= -2q_2(2x_2+q_2),
\end{align*}
upon elimination of time. Integrating this first order ODE, gives a solution
\begin{align}
 q_2 =m_2(x_2)\equiv \frac{2e^{-2x_2^2}}{\sqrt{2\pi} (1+\text{erf}\,(\sqrt{2}x_2))},\label{eq.m2Eqn}
\end{align}
with $\textnormal{erf}$ the Gauss error function \eqref{eq.errFunction},
which corresponds to $C_2=\kappa_{21}(C_1)$. 
The orbit \eqref{eq.m2Eqn} intersects $\{x_2=\eta^{-1/2}\}$ in
\begin{align}
 r_2=0,\,\hat \varepsilon=0,\,x_2 = \eta^{-1/2},\,q_2=m_2(\eta^{-1/2}).\label{eq.C2Out}
\end{align}
Hence, we obtain
\begin{lemma}\label{lem.Ma2}
 The forward flow of manifold $M_2\equiv \kappa_{21}(M_1)$ intersects $\{x_2=\eta^{-1/2}\}$ with
\begin{align}
q_2&=m_2(\eta^{-1/2}),\label{eq.q2M2}\\
\hat \varepsilon &= \left[\frac14 \ln \varepsilon^{-1} +\left( \eta^{-1} + \frac{1}{2} \ln \left(\frac{\sqrt{2\pi}}{2} \left(1+\textnormal{erf}\,\left(\sqrt{2\eta^{-1}}\right)\right)\right)\right)\right]^{-1}.
\label{eq.yM2}
 \end{align}
\end{lemma}
\begin{proof}
 The expression \eqref{eq.q2M2} follows directly from \eqref{eq.C2Out} and the fact that the right hand side of \eqref{eq.eqChartK2Tanh} is independent of $r_2$. To obtain \eqref{eq.yM2} we consider
 \begin{align}
 \frac{d\hat \varepsilon}{dx_2} &= -\hat \varepsilon^{2}(2x_2+m(x_2)),\label{eq.dydx2}
\end{align}
 obtained by inserting \eqref{eq.q2M2} into \eqref{eq.eqChartK2Tanh} 
 We can then integrate \eqref{eq.dydx2} from $$x_2=x_{20}\equiv m_2^{-1}(\nu^{-1/2}),$$ to $x_2=\eta^{-1/2}.$ For this we use the fact that
 \begin{align}
  \int m_2(x_2) dx_2 = \ln \left(1+\text{erf}\,(\sqrt{2x_2})\right) -\ln \frac{2}{\sqrt{2\pi}}= \ln \frac{e^{-2x_2^2}}{m_2(x_2)},\label{eq.intm2}
 \end{align}
 where we in the last equality have used Eq. \eqref{eq.m2Eqn}. Therefore
%
%
\begin{align*}
\hat \varepsilon^{-1} &= \frac14 \ln (\nu \varepsilon^{-1})+(\eta^{-1}-x_{20}^2)+\frac{1}{2}\left(\ln \frac{e^{-2\eta^{-1}}}{m_2(\eta^{-1/2})} - \ln \frac{e^{-2x_{20}^2}}{\nu^{-1/2}}\right)\\
&=\frac14 \ln \varepsilon^{-1} + \eta^{-1} + \frac{1}{2} \ln \frac{e^{-2\eta^{-1}}}{m_2(\eta^{-1/2})},
\end{align*}
using the initial condition from \eqref{eq.P2In} and the second equation in \eqref{eq.intm2}. Using \eqref{eq.m2Eqn} we then obtain the expression in \eqref{eq.yM2}.
\end{proof}
\begin{remark}
If $M_{a,2}$ were not independent of $r_2$, then one would have to apply regular perturbation theory in this chart. This would give rise to $\mathcal O(r_2)$-terms in \eqref{eq.q2M2} and \eqref{eq.yM2}.
\end{remark}


We illustrate the dynamics in Fig. \ref{tanhkappa2}. Finally, we move to chart $\kappa_3$. Cf. \eqref{eq.kappa2tanh} and \eqref{eq.kappa3tanh} the coordinate change between the charts $\kappa_2$ and $\kappa_3$ is given as
\begin{align}
 q_3 = x_2^{-1}q_2,\quad r_3=r_2x_2,\quad \epsilon_3=x_2^{-2},\label{eq.k2k3}
\end{align}
valid for $x_2>0$.
\begin{figure}
\begin{center}
\includegraphics[width=.7\textwidth]{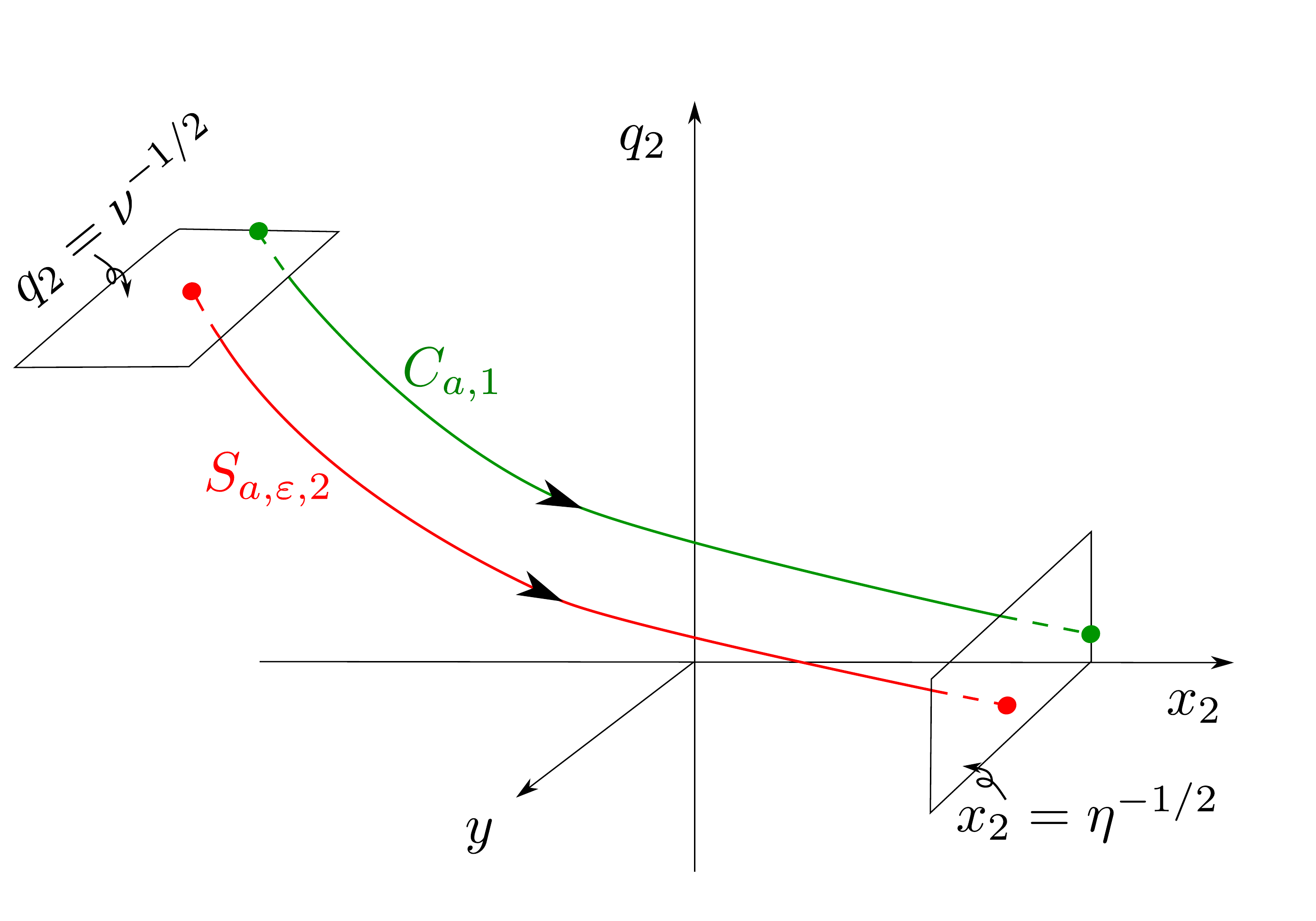}
 \caption{The dynamics within chart $\kappa_2$. }
\label{tanhkappa2}
 \end{center}
              \end{figure}

\subsection{Chart $\kappa_3$}
Inserting \eqref{eq.kappa2tanh} into \eqref{eq.extQ} gives
\begin{align}
 \dot r_3 &= r_3\epsilon_3,\label{eq.tanhK3}\\
 \dot{\hat \varepsilon} &=-\hat \varepsilon^2\left(2+q_3\right),\nonumber\\
 \dot q_3 &=-2q_3\left(2(2+q_3)+\epsilon_3\right),\nonumber\\
 \dot \epsilon_3 &=-2\epsilon_3^2, \nonumber
\end{align}
after desingularization through division of the right hand side by $r_3$. We consider the following box:
\begin{align*}
 U_3=\bigg\{(r_3,\hat \varepsilon,q_3,\epsilon_3)&\vert r_3 \in [0,\theta],\,\hat \varepsilon \in [0,\xi],\\
 &, q_3\in [0,\theta^{-1}],\,\epsilon_3\in [0,\eta]\bigg\},
\end{align*}
with side lengths $\theta$, $\xi$, $\theta^{-1}$ and $\eta$. 
The point \eqref{eq.C2Out} becomes
\begin{align}
 r_3=0,\,\hat \varepsilon=0,\,q_3 = \eta^{1/2}m_2(\eta^{-1/2}),\,\epsilon_3=\eta,\label{eq.C3In}
\end{align}
using the coordinate transformation in \eqref{eq.k2k3}. 
The set 
\begin{align}
N_3:\quad q_3=0,\label{eq.N3}
\end{align}
is an attracting (but inflowing and non-unique) center manifold of \eqref{eq.tanhK3}. We use Fenichel's normal form \cite{jones_1995} to straighten out the stable fibers:
\begin{lemma}\label{lem.FNF}
 For $\varepsilon$ sufficiently small, there exists a smooth transformation $$(r_3,\hat \varepsilon,q_3,\epsilon_3)\mapsto (r_3, \breve \varepsilon,q_3,\epsilon_3),$$ within $U_3$:
 \begin{align}
\hat \varepsilon &= \frac{\breve \varepsilon}{1+\breve \varepsilon \left(-\frac12 \ln \left(1+\frac{q_3}{2}S_3(\epsilon_3)\right)\right)},\label{eq.fiber}
 \end{align}
 with
 \begin{align}
  S_3(\epsilon_3) = \sqrt{2\pi \epsilon_3^{-1}} e^{2\epsilon_3^{-1}}\left(1-\textnormal{erf}\,\left(\sqrt{2\epsilon_3^{-1}}\right)\right)=1+\mathcal O(\epsilon_3),\label{eq.S3}
 \end{align}
transforming \eqref{eq.tanhK3} into:
\begin{align}
 \dot r_3 &=r_3\epsilon_3,\label{eq.reduced3}\\
 \dot{\breve \varepsilon} &=-2\breve \varepsilon^2,\nonumber\\
 \dot q_3 &=-2q_3\left(2(2+q_3)+\epsilon_3\right),\nonumber\\
 \dot \epsilon_3 &=-2\epsilon_3^2.\nonumber
\end{align}

\end{lemma}
\begin{proof}
 The existence of the transformation follows from Fenichel's normal form \cite{jones_1995}. It straightens out the stable fibers of $N_3$. The expression in \eqref{eq.fiber} follows by considering the $r_3=0$ system:
 \begin{align*}
 \dot{\hat \varepsilon} &=-\hat \varepsilon^2(2+q_3),\nonumber\\
 \dot q_3 &=-q_3(2(2+q_3)+\epsilon_3),\nonumber\\
 \dot \epsilon_3 &=-2\epsilon_3^2,\nonumber
 \end{align*}
and applying a transformation $$\hat \varepsilon = \frac{\breve \varepsilon}{1+\breve \varepsilon W(q_3,\epsilon_3) },$$ with $W$ having the property that
\begin{align*}
\dot{\breve{\varepsilon}} &=-2\breve{\varepsilon}^2.
\end{align*}
This gives rise to the following equation for $W(q_3,\epsilon_3)$
\begin{align*}
 q_3 = \partial_{q_3} W \dot q_3+\partial_{\epsilon_3} W \dot \epsilon_3.
\end{align*}
Using the method of characteristic we obtain a solution
\begin{align}
 W = -\frac12 \ln \left(1+\frac{q_3}{2}S_3(\epsilon_3)\right),\label{eq.Qeqn}
\end{align}
with $S_3$ as in \eqref{eq.S3}.
The smoothness of $S_3$ and the expansion in \eqref{eq.S3} follows from the following asymptotics
of $\text{erf}\,(u)$ for $u\rightarrow \infty$:
\begin{align*}
 \textnormal{erf}\,(u) = 1-\frac{e^{-u^2}}{\sqrt{\pi}u}\left(1+\mathcal O(u^{-2})\right).
\end{align*}

\end{proof}
Using \eqref{eq.k2k3} we can write \eqref{eq.Qeqn} as
\begin{align*}
 W = -\frac12 \ln \left(1+\frac{x_2^{-1} q_2}{2}S_3(x_2^{-2})\right).
\end{align*}
In particular for $q_2=m_2(x_2)$ with $m_2$ as in \eqref{eq.m2Eqn}
\begin{align*}
 W = -\frac12 \ln \left(1+\frac{1-\textnormal{erf}\,(\sqrt{2}x_2)}{1+\textnormal{erf}\,(\sqrt{2}x_2)}\right)=-\frac12 \ln \frac{2}{\textnormal{erf}\,(\sqrt{2}x_2)+1}.
\end{align*}
Therefore \eqref{eq.yM2} becomes
\begin{align}
 \breve \varepsilon^{-1} &= \frac14 \ln \varepsilon^{-1} +\bigg( \eta^{-1} + \frac{1}{2} \ln \left(\frac{\sqrt{2\pi}}{2} (1+\textnormal{erf}\,(\sqrt{2\eta^{-1}}))\right)\nonumber\\
 &+\frac12 \ln \frac{2}{1+\textnormal{erf}\,(\sqrt{2\eta^{-1}})}\bigg)\nonumber\\
 &=\frac14  \ln \varepsilon^{-1} +\left( \eta^{-1} + \frac14 \ln (2\pi)\right),\label{eq.tildeyM2}
\end{align}
in terms of $\breve \varepsilon$. 

%
%
%

\begin{lemma}\label{lem.P3}
 Consider the reduced problem \eqref{eq.reduced3}$_{q_3=0}$ on $W_{loc}^s(N_3)$ and the mapping $$P_3:\quad (r_3,\breve \varepsilon)\mapsto (\breve \varepsilon^+,\epsilon_3^+)$$ from $\{(r_3,\breve \varepsilon,\epsilon_3)\vert \epsilon_3=\eta\}$ to $\{(r_3,\breve \varepsilon,\epsilon_3)\vert r_3=\theta\}$ obtained by the forward flow. Then
 \begin{align*}
  P_3(\sqrt{\varepsilon \eta^{-1}},\breve \varepsilon) = \begin{pmatrix}
                                    \left(\breve \varepsilon(0)^{-1} +\varepsilon^{-1}\theta^2 - \eta^{-1}\right)^{-1}\\
                                     \varepsilon \theta^{-2}
                                    \end{pmatrix}
 \end{align*}
\end{lemma}
\begin{proof}
It is here easier to work with 
\begin{align}
 \breve y \equiv \breve \varepsilon^{-1}\varepsilon,\label{eq.brevey}
\end{align}
rather than $\breve \varepsilon$. This gives 
%
\begin{align*}
 \dot r_3 &=2r_3,\\
 \dot{\breve y} &=4r_3^2=2r_3\dot r_3 = \frac{d r_3^2}{dt},\\
 \dot \epsilon_3 &=-4\epsilon_3,
\end{align*}
obtained using $\varepsilon=r_3^2\epsilon_3$ and division of the right hand side by $\frac12 \epsilon_3$. 
Then with initial conditions:
\begin{align*}
 r_3(0) = \sqrt{\varepsilon \eta^{-1}},\,\epsilon_3(0) = \eta,
\end{align*}
and $r_3(T)=\theta$, $\epsilon_3(T) = \varepsilon \theta^{-2}$ we obtain a travel time of $T=\frac12 \ln \left(\theta \sqrt{\nu\varepsilon^{-1}}\right)$. Therefore
\begin{align*}
 \breve y(T) = \breve y(0) +  r_3(T)^2-r_3(0)^2 =  \breve y(0) +\theta^2 - \varepsilon \eta^{-1},
\end{align*}
or by \eqref{eq.brevey}
\begin{align*}
 \breve \varepsilon(T)^{-1} = \breve \varepsilon(0)^{-1} +\varepsilon^{-1}\theta^2 - \eta^{-1},
\end{align*}
which completes the proof.
\end{proof}
Finally, applying $P_3$ to the initial condition \eqref{eq.tildeyM2} we conclude using \eqref{eq.fiber} and $q_3(T) = \mathcal O(e^{-c/\varepsilon})$ that the slow manifold intersects $\{r_3=\theta\}$ with
\begin{align*}
\hat \varepsilon^{-1} \varepsilon &= 
 \theta^2+\frac14 \varepsilon \ln \varepsilon^{-1} +\varepsilon \left(\frac14 \ln \frac{\pi}{2}+R(\sqrt{\varepsilon})\right)\\
 &=\theta^2 + \varepsilon \left( \frac14 \ln \left(\frac{\pi}{2}\varepsilon^{-1}\right)+R(\sqrt{\varepsilon})\right),
\end{align*}
where $R(\sqrt{\varepsilon})=\mathcal O(e^{-c\varepsilon^{-1}})$ is smooth. We illustrate the geometry in Fig. \ref{tanhkappa3}. Since $y=\hat \varepsilon^{-1}\varepsilon$ by \eqref{eq.hallo} this then completes the proof of Theorem \ref{thm.tanh}.
\begin{figure}
\begin{center}
\includegraphics[width=.7\textwidth]{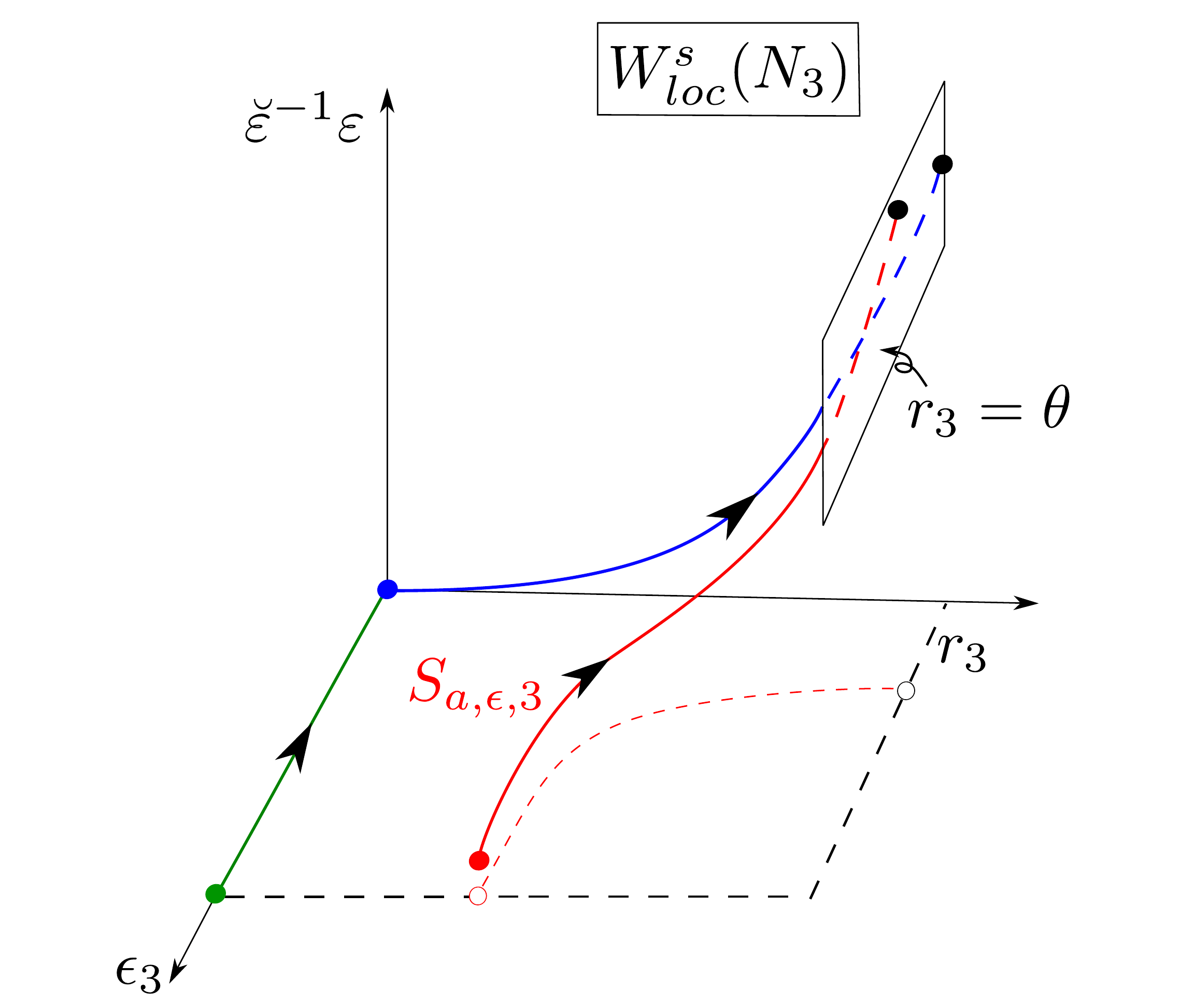}
 \caption{The dynamics within chart $\kappa_3$. }
\label{tanhkappa3}
 \end{center}
              \end{figure}

\section{Analysis of the charts $\kappa_{1-4}$ of the blowup \eqref{eq.FinalBlowup}}\label{appD}
\subsection*{Chart $\kappa_1$}
Here we prove Lemma \ref{lem.newKappa1}. The equations in chart $\kappa_1$ are
\begin{align*}
 \dot \rho_1 &= \rho_1 (x_1-1)\left[b+\frac{a\rho_1^4 y_1^2}{1+a\rho_1^2 y_1}\right],\nonumber\\
 \dot x_1 &=y_1 \left(1+\alpha \rho_1^2 y_1+\rho_1^2 x_1(x_1-1)\right)-x_1(x_1-1)\left[b+\frac{a\rho_1^4 y_1^2}{1+a\rho_1^2 y_1}\right],\nonumber\\
 \dot y_1 &=-y_1(x_1-1)\left(2\left[b+\frac{a\rho_1^4 y_1^2}{1+a\rho_1^2 y_1}\right]-y_1\rho_1^2\right).\nonumber
\end{align*}
The linearization about the origin $(\rho_1,x_1,y_1)=0$ gives $-b$, $b$ and $2b$ as eigenvalues with associated eigenvectors:
\begin{align*}
 (1,0,0),\,(0,1,0),\,(0,1,b),
\end{align*}
respectively. Using the invariance of the two planes $\{y_1=0\}$, $\{\rho_1=0\}$ we then have  $W^s(0):\,x_1=y_1=0,\,\rho_1\ge 0$ and that $W^u(0)\subset \{\rho_1=0\}$. Since $W^{uu}(0)\subset W^u(0)$ is tangent to the strong eigenvector $(0,1,b)$ at the origin this then proves $1^\circ$ of Lemma \ref{lem.newKappa1}.

%
%
%
Within $\{y_1=0\}$ we have
\begin{align*}
 \dot \rho_1 &= \rho_1 (x_1-1),\\
 \dot x_1 &= -x_1 (x_1-1),
\end{align*}
after division of the right hand side by $\left[b+\frac{a\rho_1^4 y_1^2}{1+a\rho_1^2 y_1}\right]>0$,
the dynamics leaving the hyperbolas $\rho_1 x_1=\text{const}.$ invariant. Here we find $C_1:\, x_1=1$, corresponding to $C_2$ \eqref{eq.C2} above for $\rho_1>0$, as a line of fix-points. By blowing up we have gained hyperbolicity of $C_2$ at $\rho_1=0$. Indeed, linearization about 
\begin{align}
 L_1:\quad \rho_1=y_1=0,\,x_1=1,\label{eq.L1}
\end{align}
gives $-b$ as a single non-zero eigenvalue. By center manifold theory we therefore obtain a two-dimensional center manifold $S_1$ of $L_1$ so that $S_1\cap \{y_1=0\}=C_1$. $S_1$ is an extension of $S_2$ in \eqref{eq.S2lemma} into this chart and a simple computation gives \eqref{eq.S1Expr}. Within $\{\rho_1=0\}$ we obtain the following system
\begin{align}
\dot x_1 &=y_1 -bx_1(x_1-1),\label{eq.rho1Eq0}\\
 \dot y_1 &=-2by_1(x_1-1).\nonumber
\end{align}
Here $(x_1,y_1)=(1,0)$ is an equilibrium with eigenvalues $-b$, $0$ with stable manifold $W^s:\,y_1=0,\,x_1>0$ and an inflowing, non-unique center manifold $D_1$ tangent to the eigenvector $(1,b)$:
\begin{align*}
 D_1:\quad x_1 = 1+\frac{y_1}{b}(1+\mathcal O(y_1)),\quad y_1\in [0,\chi].
\end{align*}
However, it is relatively straightforward to set up a trapping region within the $(x_1,y_1)$-plane and guide the strong unstable manifold $W^{uu}(0)$ forward and show that it is asymptotic to $(x_1,y_1)=(1,0)$. (Another approach in this case is to write the system in the $(x_1,q_1)$-variables, see \eqref{eq.Airkappa3} below, where the system can be integrated explicity). 
We can therefore take $D_1 = W^{uu}(0)\cap V_1$ for $V_1$ a small neighborhood of \eqref{eq.L1}. This proves $2^\circ$ and $3^\circ$ and gives rise to the dynamics of $\kappa_1$ illustrated in Fig. \ref{airExtra}. 

\subsection*{Chart $\kappa_2$}
In this chart we obtain the following equations
\begin{align*}
 \dot \rho_2 &= -\rho_2 y_2 \left[1+\alpha \rho_2^2 y_2+\rho_2^2(1+q_2)\right],\\
 \dot y_2 &=y_2^2 \left(2\left[1+\alpha \rho_2^2 y_2+\rho_2^2(1+q_2)\right]-\rho_2^2 (1+q_2)\right),\\
 \dot q_2 &=-q_2 \left((1+q_2)\left(b+\frac{a\rho_2^4 y_2^2}{1+a\rho_2^2 y_2}\right)-y_2\left[1+\alpha \rho_2^2 y_2+\rho_2^2(1+q_2)\right]\right).
\end{align*}
We again have two invariant planes: $\{y_2=0\}$ and $\{q_2=0\}$. Within $\{y_2=0\}$ we have 
\begin{align*}
 \dot \rho_2 &=0,\\
 \dot q_2 &=-b q_2 (1+q_2),
\end{align*}
with dynamics, in agreement with \eqref{eq.eqx2q2}, just on $q_2$, contracting towards the line of fix-points: $q_2=y_2=0$. Within $\{q_2=0\}$, on the other hand, we have
\begin{align*}
\dot \rho_2 &= -\rho_2 y_2 \left[1+\alpha \rho_2^2 y_2+\rho_2^2\right],\\
 \dot y_2 &=y_2^2 \left(2(1+\alpha \rho_2^2 y_2)+\rho_2^2\right),
\end{align*}
discovering here again $q_2=y_2=0$ as a set of fix-points. Division of the right hand side by  $y_2$ gives
\begin{align*}
\dot \rho_2 &= -\rho_2\left[1+\alpha \rho_2^2 y_2+\rho_2^2\right],\\
 \dot y_2 &=y_2 \left(2(1+\alpha \rho_2^2 y_2)+\rho_2^2\right),
\end{align*}
where $\rho_2=y_2=0$ now becomes the only equilibrium with eigenvalues $-1$ and $2$; hence a saddle. Combining this gives rise to the dynamics of $\kappa_2$ illustrated in Fig. \ref{airExtra}.  

\subsection*{Chart $\kappa_3$}
This chart is simple. Within $\rho_3=0$, for example, we obtain 
\begin{align}
 \dot x_3 &=1,\label{eq.Airkappa3}\\
 \dot q_3 &= bq_3(x_3-q_3).\nonumber
\end{align}
Integrating this one-dimensional system gives rise to the dynamics of $\kappa_3$ illustrated in Fig. \ref{airExtra}. Each orbit of  \eqref{eq.Airkappa3} has a vertical asymptote $x_3=\text{const}.$ The strong unstable $W^{uu}$ from chart $\kappa_1$ corresponds to the special solution
\begin{align*}
q_3(x_3) = i\sqrt{\frac{2}{b\pi}}\exp\left(\frac{b}2 x_3^2\right)\textnormal{erf}\left(i\sqrt{\frac{b}{{2}}}x_3\right)^{-1}, 
\end{align*}
in this chart, which blows up $q_3\rightarrow \infty$ as $x_3\rightarrow 0^+$, satisfying (using simple asymptotics of the Gauss error function $\textnormal{erf}$ \eqref{eq.errFunction}) 
\begin{align*}
 q_3(x_3) = \frac{1}{bx_3}+\mathcal O(x_3),\quad q_3(x_3) = x_3-\frac{1}{bx_3}+\mathcal O(x_3^{-3}),
\end{align*}
for $x_3\rightarrow 0^+$ and $x_3\rightarrow \infty$, respectively.

\subsection*{Chart $\kappa_4$}
In this chart we obtain the following equations
\begin{align*}
 \dot \rho_4 &= \rho_4 y_4 \left[1+\alpha \rho_4^2 y_4+\rho_4^2 (1-q_4)\right],\\
 \dot y_4 &= -y_4^2 \left(2\left[1+\alpha \rho_4^2 y_4+\rho_4^2 (1-q_4)\right]-\rho_4^2(1-q_4)\right),\\
  \dot q_4 &= q_4 \left(1-q_4-y_4 \left[1+\alpha \rho_4^2 y_4+\rho_4^2 (1-q_4)\right]\right).
\end{align*}
Here we re-discover the fix point $L_1$, the line of fix points $C_1\subset \{y_1=0\}$ and the inflowing center $D_1\subset \{\rho_1=0\}$ from chart $\kappa_1$ as 
\begin{align*}
L_4:\quad &\rho_4 = 0,\,y_4=0,\,q_4 = 1,\\
 C_4:\quad &\rho_4>0,\,y_4=0,\,q_4 = 1,\\
 D_4:\quad &\rho_4 =0,\,q_4 = 1-\frac{1}{b}y_4(1+\mathcal O(y_4)) ,\quad y_4\in [0,\chi],
 \end{align*}
 respectively, with $\chi$ sufficiently small. The invariant set $\{q_4=0\}$ is unstable for $y_4$ and $\rho_4$ small. Initial conditions from chart $\kappa_3$ (like the blue orbit in Fig. \ref{airExtra}) enter $\kappa_4$ close to this set and, in agreement with Proposition \ref{finalProp}, follow this for some time before contracting towards $S_4$; the attracting center manifold of $L_4$. This completes the picture in Fig. \ref{airExtra}. 

\bibliography{refs}
\bibliographystyle{plain}
 \end{document}